\documentclass[a4paper,oneside]{article}
\usepackage[top=4cm,bottom=3.5cm,left=3.5cm,right=3.5cm]{geometry}
\usepackage[T1]{fontenc}
\usepackage[utf8]{inputenc}
\usepackage[UKenglish]{babel}
\usepackage{amsthm}
\usepackage{amsmath,amssymb}
\usepackage{nicefrac}
\usepackage{mathtools}
\usepackage{mathdots}
\usepackage{hyperref}
\usepackage{centernot}
\usepackage[inline]{enumitem}
\newlist{choices}{enumerate*}{1}
\setlist[choices]{itemjoin = \hspace{1cm}, label=\arabic*.}

\DeclareMathOperator{\smod}{smod}
\DeclareMathOperator{\st}{st}
\DeclarePairedDelimiter{\set}{\{}{\}}
\DeclarePairedDelimiter{\abs}{\lvert}{\rvert}

\newcommand{\bla}[4]{{#1}_{#2}#3\ldots#3{#1}_{#4}}

\usepackage{thmtools}
\usepackage{cleveref}
\usepackage{nameref}
\Crefname{rem}{Remark}{Remarks}
\Crefname{alphthm}{Theorem}{Theorems}
\Crefname{thm}{Theorem}{Theorems}
\Crefname{equation}{Equation}{Equations}

\usepackage[backend=bibtex,style=numeric,maxbibnames=99,sorting=nyt]{biblatex}
\renewbibmacro{in:}{}

\newcommand{\N}{\mathbb{N}}

\newcommand{\Z}{\mathbb{Z}}
\newcommand{\from}{:}
\newcommand{\inverse}{^{-1}}
\newcommand{\nepfin}{\mathcal P_{\mathrm{fin}}(\mathbb N)\setminus\set\emptyset}
\newcommand{\fcol}{\mathbb N=\bla C1\cup r}
\renewcommand{\star}{{}^\ast}
\renewcommand{\phi}{\varphi}
\renewcommand{\epsilon}{\varepsilon}

\theoremstyle{definition}
\newtheorem{thm}{Theorem}[section]
\newtheorem{lemma}[thm]{Lemma}
\newtheorem{rem}[thm]{Remark}
\newtheorem{pr}[thm]{Proposition}
\newtheorem{co}[thm]{Corollary}
\newtheorem{fact}[thm]{Fact}
\newtheorem{defin}[thm]{Definition}
\newtheorem{eg}[thm]{Example}
\newtheorem{notation}[thm]{Notation}
\newtheorem{problem}[thm]{Problem}
\newtheorem{question}[thm]{Question}
\Crefname{question}{Question}{Questions}
\Crefname{fact}{Fact}{Facts}
\newtheorem{claim}{Claim}[thm]

\newtheorem{alphthm}{Theorem}

\newtheorem{alphco}[alphthm]{Corollary}

\newtheoremstyle{named}{}{}{}{}{\bfseries}{.}{.5em}{\thmnote{#3}}
\theoremstyle{named}

\let\oldqed\qedsymbol

\newcommand{\qedclaim}{\mbox{$\underset{\textsc{claim}}{\oldqed}$}}
\makeatletter
\newenvironment{claimproof}[1][\it Proof of Claim]{
\let\qedsymbol\qedclaim
  \par
  \pushQED{\qed}%
  \normalfont \topsep6\p@\@plus6\p@\relax
  \trivlist
\item[\hskip\labelsep
  \upshape
  #1\@addpunct{.}]\ignorespaces
}{%
  \popQED\endtrivlist\@endpefalse
}
\makeatother
\let\qedsymbol\oldqed

\usepackage{titling}
\usepackage[affil-it]{authblk}
\usepackage{orcidlink}

\newcommand{\email}[1]{\href{mailto:#1}{\texttt{#1}}}

\makeatletter
\newcommand{\subjclass}[2][2020]{%
  \let\@oldtitle\@title%
  \gdef\@title{\@oldtitle\footnotetext{\hspace*{-2em}#1 \emph{Mathematics subject classification.} #2}}%
}
\newcommand{\keywords}[1]{%
  \let\@@oldtitle\@title%
  \gdef\@title{\@@oldtitle\footnotetext{\hspace*{-2em}\emph{Keywords:} #1.}}%
}

\makeatother

\author[1]{Mauro Di Nasso\orcidlink{0000-0001-6103-9775}%
  \thanks{email: \email{mauro.di.nasso@unipi.it}}}

\author[2]{Lorenzo Luperi Baglini\orcidlink {0000-0002-0559-0770}%
  \thanks{email: \email{lorenzo.luperi@unimi.it}}}

\author[ ]{Rosario Mennuni\orcidlink{0000-0003-2282-680X}%
  \thanks{email: \email{R.Mennuni@posteo.net}}}

\author[3]{Mariaclara Ragosta\orcidlink{0009-0004-6641-4676}%
  \thanks{email: \email{mariaclara@kam.mff.cuni.cz}}}

\author[2,4]{Alessandro Vegnuti\orcidlink{0009-0005-2490-720X}%
  \thanks{email: \email{alessandro.vegnuti@unimi.it}}}

\affil[1]{\small Dipartimento di Matematica, Universit\`a di Pisa, Largo Bruno Pontecorvo 5, 56127 Pisa, Italy}
\affil[2]{\small Dipartimento di Matematica, Università  di Milano, Via Saldini 50, 20133 Milano, Italy}
\affil[3]{\small Department of Applied Mathematics (KAM), Charles University, Malostranské náměstí 25, Praha 1, Czech
  Republic}
\affil[4]{Mathematisches Institut, Albert-Ludwigs-Universität Freiburg, D-79104 Freiburg,
Germany}

\title{Monochromatic sums and quotients in $\mathbb N$}



\begin{document}

\maketitle

\begin{abstract}
We prove partition regularity of the configuration $x,y,x+y,y/x$ in a strong infinitary form that extends Hindman's Theorem. We study the related issue of partition regularity of configurations involving products of a degree one polynomial in $x$ with one in $y$, reducing the general problem to a handful of special cases.
\end{abstract}

\section*{Introduction}

\paragraph{Sums, products and colours.} A much studied problem in the combinatorics of the natural numbers is that of monochromaticity of arithmetic configurations. The archetypal result in this area is Schur's Theorem~\cite{schurUberKongruenz}, saying that the pattern $x,y,x+y$ is \emph{partition regular}: whenever the natural numbers\footnote{In order to avoid trivialities, throughout the paper we convene that $0\notin \mathbb N$.} are finitely coloured, i.e.\ partitioned in finitely many pieces, there are some $x,y\in \mathbb N$ such that $x,y,x+y$ lie in the same colour.

In this context, a cornerstone result is Hindman's infinitary extension of Schur's Theorem~\cite{hindmanFiniteSumsSequences1974}: in every finite colouring of $\mathbb N$, one may find an infinite sequence such that all sums of finitely many of its terms are monochromatic. The same holds when the sum is replaced by an arbitrary associative operation~\cite[Theorem~5.8]{hindmanAlgebraStoneCechCompactification2011}, and in particular the pattern $x,y,xy$ is partition regular. 

However, despite the thorough literature now available on the problem of partition regularity of Diophantine equations\footnote{See the introduction to \cite{di2018ramsey} for a survey.}, whether  $x,y,x+y,xy$ is partition regular remains one of the most long-standing open problems in the area. About this pattern, what is known at present essentially amounts to the following.
  \begin{itemize}
  \item Partition regularity holds for colourings of $\mathbb Q$, see \cite{bowenMonochromaticProductsSums2024}.
  \item Monochromatic $x,y,x+y,xy$ may be found in every colouring of $\mathbb N$ with $2$ colours, see~\cite{hindmanPartitionsSumsProducts1979}. In fact, in such a colouring one may find monochromatic sums and products of arbitrarily large finite size, see~\cite{bowenMonochromaticProductsSums2025}.
  \item Over $\mathbb N$, the pattern $x,x+y,xy$ is partition regular, see \cite{moreiraMonochromaticSumsProducts2017}.
  \item Over $\mathbb N$, the Hindman version of partition regularity fails, already at the level of pairs. Namely, there is a finite colouring of $\mathbb N$ admitting no infinite sequence $(x_i)_{i\in \mathbb N}$ such that all $x_i+x_j$ and all $x_i x_j$, for $i\ne j$, lie in the same colour. See~\cite{hindmanPartitionsPairwiseSums1984a}, or~\cite{dinassoRamseysWitnesses2025} for a compact proof and a systematic study of similar problems. Over $\mathbb Q$, the Hindman version is still open.
  \end{itemize}

\paragraph{Results.} In this paper we prove that sums and ratios are, instead, much more well-behaved from this point of view, as Hindman's Theorem may be extended to accommodate ratios of sums coming from a fixed sequence.
\begin{alphthm}[{\Cref{thm:fsandratios}}]\label{athm:fsrations}
  For every finite colouring $\fcol$ there are $m\le r$ and an increasing sequence $(x_i)_{i\in \mathbb N}$ such that, for all $k<\ell\in \mathbb N$ and all $\bla i1<\ell$, the colour $C_m$ contains all 
\begin{equation*}
\bla x{i_1}+{i_k}\quad\text{ and }\quad \frac{\bla x{i_{k+1}}+{i_\ell}}{\bla x{i_1}+{i_k}}.
\end{equation*}
\end{alphthm}
 In particular, this solves~\cite[Problem~6.3(3)]{dinassoRamseysWitnesses2025}, and gives that the pattern $x,y,x+y,y/x$ is partition regular over $\mathbb N$. In fact, so is every pattern $x,y,x+y,q\cdot y/x$ with $q\in \mathbb Q_{>0}$ (\Cref{co:fsratiosrescaled}) but, besides this, \Cref{athm:fsrations} is sharp, in the sense made precise by \Cref{thm:qcd}.

In fact, \Cref{thm:fsandratios} tells us more.  A simple change of variables ($a=x$ and $b=y/x$) transforms $x,y,x+y,y/x$ into $a,b,ab,a(b+1)$. Partition regularity of the latter was established by Goswami~\cite{goswamiMonochromaticTranslatedProduct2024}, answering a question of Sahasrabudhe~\cite{sahasrabudheExponentialPatternsArithmetic2018}, that in fact only asked about the pattern $a,b,a(b+1)$. Our \Cref{thm:fsandratios} also provides a strong infinitary version of Goswami's result and covers, as a baby case, the configuration
\[
a,\quad b,\quad c,\quad ab,\quad bc,\quad abc,\quad a+ab,\quad  a+abc,\quad  b+bc,\quad  ab+abc,\quad   a+ab+abc.
\]

The partition regularity of $a,b,a(b+1)$ is somewhat unexpected\footnote{In fact, Sahasrabudhe conjectured its failure while stating his question in~\cite{sahasrabudheExponentialPatternsArithmetic2018}.}, let alone that of $a,b,ab,a(b+1)$. One may therefore wonder what other partition regular configurations are out there involving products of two shifts. After proving the aforementioned results in \Cref{sec:qos}, the remainder of the present work addresses a more general version of this problem. To avoid trivialities, below we strengthen the notion of partition regularity by requiring the existence of infinitely many monochromatic solutions, cf.~\Cref{notation}.\ref{point:nontrivialsol}.

In \Cref{sec:prodshifts} we consider configurations $x,y,q(ax+n)(by+m)$, for $q\in \mathbb Q_{>0}$, $a,b\in \mathbb N$ and $n,m\in \mathbb Z$. The case $n=m=0$ is known to be partition regular (\Cref{rem:qxy}). When precisely one of $n,m$ is zero, say $n$, the general problem remains open but, for example, for every $b,m$ the generalised Sahasrabudhe pattern $x,y,\frac 1m x(by+m)$ is partition regular (\Cref{thm:nld}). Our next main result is a necessary condition for partition regularity in the case $n\ne 0\ne m$, involving consecutive squares and (the doubles of) consecutive triangular numbers. 

\begin{alphthm}[{\Cref{thm:squares}}]\label{athm:cosq}
  If  $nm\ne 0$ and the pattern
  \begin{equation*}
    x,\quad y,\quad q(ax+n)(by+m)
  \end{equation*}
 is partition regular, then there is $t\in \mathbb Z$ such that
  \begin{enumerate}[label=(\alph*)]
  \item\label{case:conssqrintro} $qam=t^2$ and $qbn=(t+1)^2$, or
  \item\label{case:consprodintro} $qam=t(t-1)$ and $qbn=t(t+1)$.
  \end{enumerate}
\end{alphthm}

This shows that, for instance, the pattern $x,y,(x+1)(y+2)$ is not partition regular. This necessary condition is certainly not sufficient, e.g.\ the configuration $x,y,(x+1)(4y+1)$ falls in case~\ref{case:conssqrintro} but is not partition regular, as witnessed by colouring natural numbers by their parity. This does not happen by chance, and is in fact a special case of the divisibility condition provided by our next theorem, which also gives a canonical form for patterns satisfying the conclusion of \Cref{athm:cosq}.

\begin{alphthm}\label{athm:newC}
  Suppose there is $t\in \mathbb Z$ satisfying the conclusion of \Cref{athm:cosq} (possibly with $nm=0$).   The pattern
\begin{equation*}
      x,\quad y,\quad q(ax+n)(by+m)
    \end{equation*}
  is partition regular if and only if the pattern
    \begin{equation*}
      x,\quad y,\quad (x+qbn)(y+qam)
    \end{equation*}
is partition regular and either
\begin{enumerate}[label=(\alph*)]
\item\label{point:newCa} $qam=t^2$,  $qbn=(t+1)^2$ and $(am+bn)/2ab -1/2qab\in \mathbb Z$,  or
  \item\label{point:newCb} $qam=t(t-1)$, $qbn=t(t+1)$ and $(am+bn)/2ab \in \mathbb Z$ .
\end{enumerate}
\end{alphthm}

In \Cref{sec:twopos} we study 4-piece configurations. As it turns out, very few of them have a chance to be partition regular. 

\begin{alphthm}\label{athm:4pieces}
   Let $q_i\in \mathbb Q_{>0}$, $a_i,b_i\in \mathbb N$ and $n_i,m_i\in \mathbb Z$. If the pattern
    \begin{equation*}
  x,\quad y,\quad q_1(a_1x+n_1)(b_1y+m_1),\quad q_2(a_2x+n_2)(b_2y+m_2)
\end{equation*}
is partition regular, then 
\begin{enumerate}
\item\label{case:really3} $q_1(a_1x+n_1)(b_1y+m_1)= q_2(a_2x+n_2)(b_2y+m_2)$, or
\item\label{case:cross} $q_1(a_1x+n_1)(b_1y+m_1)= q_2(b_2x+m_2)(a_2y+n_2)$, or
\item\label{case:0mixed} the pattern is $x, y, q xy, qxy + x$ or the symmetric $x, y, q xy, q xy + y$, with $q\coloneqq q_1a_1b_1=q_2a_2b_2$.
   \end{enumerate}
\end{alphthm}

We point out that, in case~\ref{case:0mixed}, partition regularity  does indeed hold (\Cref{co:qxy}). As for $n$-piece configurations with $n\ge 5$, observe that every subconfiguration of a partition regular one is itself partition regular. This, together with \Cref{athm:4pieces}, yields that, up to rescaling, there is only one 5-piece configuration with a chance to be partition regular, and that there are no 6-piece partition regular configurations at all.

\begin{alphco}\label{aco:5pezzi}
    Assume that the set $\set{x,y,q_i(a_ix+n_i)(b_iy+m_i)\mid i\leq k}$ has cardinality $k+2\geq 5$ and that the pattern is partition regular. Then $k= 3$ and the pattern has the form $x, y, qxy, qxy + x, qxy + y$. Moreover, the partition regularity of this pattern is equivalent to that of $x, y, xy, xy + x, xy + y$.
\end{alphco}

Furthermore, we have the following 4-piece analogue of \Cref{athm:newC}, which holds without assuming a priori that we are in the conclusion of \Cref{athm:cosq}.

\begin{alphthm}\label{athm:div}
  Let $q\in \mathbb Q_{>0}$, $a,b\in \mathbb N$ and $m,n\in \mathbb Z$.
Assume that $q(ax+n)(by+m)\ne q(bx+m)(ay+n)$. The pattern
  \begin{equation*}
      x,\quad y,\quad q(ax+n)(by+m), \quad q(bx+m)(ay+n)
    \end{equation*}
    is partition regular if and only if the pattern
    \begin{equation*}
      x,\quad y,\quad (x+qbn)(y+qam),\quad  (x+qam)(y+qbn)
    \end{equation*}
is partition regular and there is $t\in \mathbb Z$ such that point~\ref{point:newCa} or point~\ref{point:newCb} of \Cref{athm:newC} holds.
\end{alphthm}
We leave open the partition regularity of patterns that are not excluded by the theorems above, notable examples being  $x, y, x(y+1),(x+1)y$  and  $x, y, x(y+2),(x+2)y$. We collect these and other questions in \Cref{sec:froq}.

\paragraph{Methodology.} All results of this paper were originally obtained by nonstandard-analytic methods. The key observation is that, in model-theoretic parlance, partition regularity of a formula $\phi(\bla x1,k)$ is equivalent to the existence, in an elementary extension $\star \mathbb N$ of $\mathbb N$, of a solution of the formula
with all coordinates of the same type (\Cref{thm:NScharPR}). As well-known, this is in turn equivalent to the existence of an ultrafilter on $\mathbb N^k$ containing the set of solutions of $\phi(\bla x1,k)$ in $\mathbb N$ and projecting on every coordinate to the same ultrafilter on $\mathbb N$. The basics of these techniques are recalled in \Cref{sec:NS}.

With some effort, our proofs may be translated in standard ultrafilter terms, making no use of nonstandard analysis nor of model theory.  We give such a presentation of the proof of \Cref{athm:fsrations}, that makes use of a variant of the Milliken--Taylor Theorem from \cite{bergelsonPolynomialExtensionsMillikenTaylor2014} involving idempotent ultrafilters.

Nevertheless, when dealing with proofs of necessary conditions, the nonstandard approach yields more streamlined and---we believe---conceptually clearer proofs. It will become evident in \Cref{sec:prodshifts,sec:twopos} that there are obstructions of $p$-adic nature to the partition regularity of certain patterns. Said obstructions are particularly visible from the nonstandard viewpoint; we exemplify this by proving \Cref{thm:qcd} by these methods first, and then translating the proof in standard terms in an appendix, \Cref{sec:appendix}.

\paragraph{Funding}
The authors were supported by the project PRIN 2022 ``Logical methods in combinatorics'', 2022BXH4R5, Italian Ministry of University and Research (MUR). M.~Ragosta is supported by project 25-15571S of the Czech Science Foundation (GAČR). This work has been supported by Charles University Research Centre programme No.UNCE/24/SSH/026. We acknowledge the MUR Excellence Department Project awarded to the Department of Mathematics, University of Pisa, CUP I57G22000700001. M.~Di Nasso is a member of the INdAM research group GNSAGA.

\section{Quotients of sums}\label{sec:qos}

This section is devoted to the proof of \Cref{athm:fsrations}. We assume familiarity with basic notions around ultrafilters, and refer the reader to \cite{hindmanAlgebraStoneCechCompactification2011} for an extensive treatment.

\begin{notation}\label{notation} We adopt the following conventions.
  \begin{enumerate}
  \item\label{point:nontrivialsol} If we say that the pattern $f_1(\bla x1,n), \ldots, f_k(\bla x1,n)$ is \emph{partition regular} (or simply \emph{PR}) we mean that, for every finite colouring of $\mathbb N$, the set of monochromatic $k$-tuples of the form $(f_1(\bla a1,n),\ldots, f_k(\bla a1,n))$ is infinite. E.g., in every colouring of $\mathbb N$, there is a monochromatic solution of  $x,y,3x-2y,5y/x$ obtained by setting $x=y=5$, but as we will see in \Cref{thm:qcd} this pattern is not PR in the sense mentioned above.  We do this in order to avoid having to handle separately trivial cases such as the one just mentioned.
  \item By definition, saying that the pattern $x,y,f(x,y)$ is PR is the same as saying that the equation $z=f(x,y)$ is PR. We use the two terminologies interchangeably.
  \item By $D(x,y)$ we denote the function $\mathbb N^2\to \mathbb N$ sending $(x,y)$ to $y/x$ if $x\mid y$, and to $1$ otherwise.
  \item  As usual, $\operatorname{FS}(x_i\mid  i\in \mathbb N)$ denotes $\set{\bla x{i_1}+{i_k}\mid k\in \mathbb N, \bla i1<k}$.
  \item If $F,G\in \nepfin$, by $F<G$ we mean $\max F<\min G$.
  \item If $u\in \beta \mathbb N^k$ is an ultrafilter and $f\from \mathbb N^k\to \mathbb N$ is a function, we denote by $f(u)$ the pushforward $\set{A\subseteq \mathbb N \mid f^{-1}(A)\in u}$ of $u$ along $f$.
  \end{enumerate}
\end{notation}

\begin{rem}\label{rem:idempsd}
Let $u\in \beta \mathbb N$ be an additive idempotent.
\begin{enumerate}
\item\label{item:1}  The ultrafilter $u$ contains every $n \mathbb N$.
\item The ultrafilter $u$ is \emph{self-divisible}, that is, $\set{(a,b)\mid a\text{ divides } b}\in u\otimes u$.  
\item The ultrafilter $D(u\otimes u)$ contains every $n \mathbb  N$. In particular, it is not the principal ultrafilter on $1$.
\end{enumerate}
\end{rem}
\begin{proof}\*
  \begin{enumerate}
  \item This is well-known (and easy to prove).
  \item This is~\cite[Example~5.1(8)]{dinassoSelfdivisibleUltrafiltersCongruences2025}, but it can be easily shown as follows. We have $\set{(a,b)\mid  a\text{ divides } b}\in u\otimes u$ if and only if $\set{n\mid  n \mathbb N\in u}\in u$. It now suffices to apply point~\ref{item:1}.
  \item We have $n \mathbb N\in D(u\otimes u)$ if and only if $\set{(a,b)\mid n\text{ divides } b/a}\in u\otimes u$, if and only if $\set{a\mid \set{b\mid b\in an \mathbb N} \in u}\in u$ and we conclude by point~\ref{item:1}.\qedhere
  \end{enumerate}
\end{proof}

\begin{fact}\label{fact:bpemt}
Let $A\subseteq \mathbb N^2$. There is a sequence $(x_i)_{i\in \mathbb N}$ such that \[\set*{\left(\sum_{i\in F} x_i, \sum_{i\in G} x_i\right)\Biggm| F,G\in \nepfin, F<G}\subseteq A\] if and only if there is an idempotent  $u\in \beta \mathbb N$ such that $A\in u\otimes u$.
\end{fact}
\begin{proof}
This is a special case of~\cite[Theorem~1.17]{bergelsonPolynomialExtensionsMillikenTaylor2014}.
\end{proof}

\begin{thm}\label{thm:fsandratios}
For every finite colouring $\fcol$ there are $m\le r$ and sequences $(x_i)_{i\in \mathbb N}$ and $(y_j)_{i\in \mathbb N}$ such that
\begin{equation}
  \label{eq:1}
  \operatorname{FS}(x_i\mid  i\in \mathbb N) \cup   \set*{\frac{\sum_{i\in G} x_i}{\sum_{i\in F} x_i}
    \Biggm| 
    F,G\in \nepfin,  F<G
  }\subseteq C_m
\end{equation}
  and 
  \begin{equation}
    \label{eq:2}
      \set[\Bigg]{\sum_{\ell\in F}\prod_{i=k}^\ell y_i \Biggm| F\in \nepfin, k\le \min F
      }\subseteq C_m.
  \end{equation}
  Moreover, if $u$ is any additively idempotent ultrafilter, we may take as $C_m$ the colour belonging to $D(u\otimes u)$.
\end{thm}
\begin{proof}  
  Let $u$ be an additively idempotent ultrafilter and let $C_m$ be as in the ``moreover'' part. By \Cref{rem:idempsd} the set $\set{(a,b)\mid  a\text{ divides } b}$ belongs to $u\otimes u$.  Apply \Cref{fact:bpemt} to the set $D\inverse(C_m)\cap \set{(a,b)\mid  a\text{ divides } b}$, obtaining a sequence $(z_n)_{z\in \mathbb N}$ such that, whenever $F,G\in \mathcal P_{\mathrm{fin}}(\mathbb N)\setminus\set\emptyset$ are such that $F<G$, we have
 \[
   \frac{\sum_{i\in G}z_i}{\sum_{i\in F}z_i}\in C_m.
 \]
  To conclude the proof, we simply set $x_n\coloneqq z_{n+1}/z_1$ and $y_n\coloneqq z_{n+1}/z_n$. Then,  for every     $F,G\in \nepfin$,  we have
\[
    \sum_{i\in F} x_i=\frac{\sum_{i\in F} z_{i+1}}{z_1} \qquad\text{ and }\qquad
    \frac{\sum_{i\in G} x_i}{\sum_{i\in F} x_i}=\frac{\frac 1{z_1}\sum_{i\in G} z_{i+1}}{\frac 1{z_1}\sum_{i\in F} z_{i+1}}=\frac{\sum_{i\in G} z_{i+1}}{\sum_{i\in F} z_{i+1}}.
\]
 It follows that
  \begin{multline*}
    \operatorname{FS}(x_i\mid  i\in \mathbb N) \cup   \set*{\frac{\sum_{i\in G} x_i}{\sum_{i\in F} x_i}
      \Biggm|
      F,G\in \nepfin,  F<G
    }
\\
    \subseteq \set*{\frac{\sum_{i\in G} z_i}{\sum_{i\in F} z_i}
      \Biggm|
      F,G\in \nepfin,  F<G
    }
    \subseteq C_m.
  \end{multline*}
  Moreover, if $F\in \nepfin$ and $k\le \min F$, we analogously have
  \[
    \sum_{\ell\in F}\prod_{i=k}^\ell y_i =\sum_{\ell\in F}\prod_{i=k}^\ell \frac{z_{i+1}}{z_i} =\sum_{\ell \in F}\frac {z_{\ell+1}}{z_k}=\frac{\sum_{\ell \in F}z_{\ell+1}}{z_k}\in C_m.\qedhere
  \]
\end{proof}
See \cite[Section~17.3]{hindmanAlgebraStoneCechCompactification2011} for other results involving sums of products of elements from an infinite sequence. Note that the configurations there are things of the form e.g.\ $x_1+x_2x_3+x_4+x_8x_{11}$, while those in the previous theorem look like $x_1+x_1x_2x_3+x_1x_2x_3x_4$.

\begin{co}[Goswami]\label{co:goswami}
  The pattern $x,y,xy,x(y+1)$ is PR.
\end{co}
\begin{proof}
  By \Cref{thm:fsandratios} the pattern $a,b,a+b,b/a$ is PR. Apply the change of variables $x\coloneqq a$, $y \coloneqq b/a$.
\end{proof}

\begin{co}\label{co:fsratiosrescaled}
  For every finite colouring $\fcol$ and $q\in \mathbb Q_{>0}$ there are $m\le r$ and a sequence $(x_i)_{i\in \mathbb N}$ such that
  \begin{equation*}
  \operatorname{FS}(x_i\mid  i\in \mathbb N) \cup   \set*{q\cdot\frac{\sum_{i\in G} x_i}{\sum_{i\in F} x_i}
    \Biggm|
    F,G\in \nepfin,  F<G
  }\subseteq C_m.
\end{equation*}
\end{co}
\begin{proof}
  If $v$ is an ultrafilter witnessing that \Cref{thm:fsandratios} holds, then $q\cdot v\in \beta \mathbb Q$ contains $\mathbb N$ and witnesses that the conclusion holds.

If the reader prefers, they may directly run the proof of \Cref{thm:fsandratios} with $D(u\otimes u)$ replaced by $q\cdot D(u\otimes u)$.
\end{proof}

\Cref{co:fsratiosrescaled} above shows that one can add coefficients in front of the ratios. A natural question is whether the same holds for the linear term. The answer is negative.

\begin{thm}\label{thm:qcd}
  Let $q\in \mathbb Q_{>0}$ and $c,d\in \mathbb Z\setminus\set0$. Then  $x,y,cx+dy,q \cdot y/x$ is PR if and only if $c=d=1$. 
\end{thm}

We will give a nonstandard proof of this in \Cref{sec:NS}, and a standard one in \Cref{sec:appendix}.

\begin{rem}\*
  \begin{enumerate}
  \item As a special case of~\eqref{eq:1}, we see that sums, divisions and projections are Ramsey partition regular, in the sense of \cite{dinassoRamseysWitnesses2025}, in particular answering in the positive \cite[Problem~6.3(3)]{dinassoRamseysWitnesses2025} (which only asked for sums and divisions).
  \item A special case of \Cref{thm:qcd} is that differences, divisions \emph{and projections} are not PR, let alone Ramsey PR.
  \end{enumerate}
\end{rem}
Ramsey partition regularity of differences and products and of differences and divisions (the other points of \cite[Problem~6.3]{dinassoRamseysWitnesses2025}), as well as the long-standing problem of partition regularity of sums, products and projections, remain open.

\section{A quick review of nonstandard methods}\label{sec:NS}

In what follows, we will use some methods and terminology coming from nonstandard analysis that have proven very well-suited to study the partition regularity of Diophantine equations. 
In this section, we collect all the results we need; we refer to \cite{di2015hypernatural,DGL,baglini2012hyperintegers} for extended presentations of this approach.

We work in a nonstandard extension $\star\mathbb{R}$ of $\mathbb {R}$ (which we assume to be sufficiently saturated). In particular, we are interested in the substructure induced on $\star \N\subseteq\star \mathbb  R$. In such a setting, it is possible to introduce the notion of $u$-equivalence.

\begin{defin}\label{1 def sim}
    Let $\alpha, \beta\in\star\N$. We say that $\alpha,\beta$ are $u$-equivalent, and write $\alpha\sim \beta$, if for all $A\subseteq \N$ \[\alpha \in\star A \iff \beta \in \star A.\]
\end{defin}

Given $\alpha\in\star \mathbb N$, the set $\set{A\subseteq \mathbb N\mid \alpha\in \star A}$ is an ultrafilter, and every ultrafilter arises in this way by saturation. From a model-theoretical perspective, being $u$-equivalent amounts precisely to having the same type over $\emptyset$, in the language with, for each $k$, a predicate for every subset $A\subseteq \N^k$.

The main properties of $u$-equivalence that we will use are listed here; the interested reader can find a proof in \cite[Section 11.2]{di2015hypernatural}.

\begin{fact}\label{pr:PropSim}
    Let $\alpha, \beta\in\star\N$ and $f:\N\rightarrow \N$. 
    \begin{enumerate}
                \item \label{point:propsim1}If $\alpha \sim\beta$ then $f(\alpha)\sim f(\beta)$.
        \item If $\alpha \sim f(\alpha)$ then $\alpha =f(\alpha)$; in particular, if  $\alpha\sim\beta$ then $\alpha=\beta$ or $|\alpha-\beta|$ is infinite.
        \item If $\beta\sim f(\alpha)$ then there exists $\gamma\sim \alpha$ such that $\beta=f(\gamma)$.
        \item \label{point:propsim4} If $\alpha\sim\beta$ and $\alpha\in \N$ then $\alpha=\beta$.
    \end{enumerate}
\end{fact}

Recall that two elements $\alpha,\beta\in \star \mathbb R$ are in the same \emph{Archimedean class} if $\frac{1}{n}\abs \alpha<\abs \beta<n \abs\alpha$ for some $n\in \N$. This is equivalent to saying that the quotient $\alpha/\beta\in \star \mathbb{R}$ is in the Archimedean class of 1. In this case, we will denote by $\st(\alpha/\beta)$ the unique real such that $\alpha/\beta - \st(\alpha/\beta)$ is infinitesimal. 

In our proofs, we will repeatedly use the following fact, whose proof can be found in \cite[Lemma 5.5.(1)]{dinassoRamseysWitnesses2025}, relating $u$-equivalence and Archimedean classes. See \Cref{lemma:starchclass} for a standard version (and a standard proof) of the same result.

\begin{fact}\label{pr:ClassiArchimedee} 
Let $\alpha\sim\beta$ be infinite and in the same Archimedean class. Then $\st\left({\alpha}/{\beta}\right)=1$.
\end{fact}

It was proved in \cite[Theorem 2.2.9]{baglini2012hyperintegers} that the notion of $u$-equivalence allows us to rephrase partition regularity in nonstandard terms. For patterns, it reads as follows.

\begin{fact}\label{thm:NScharPR}
    Given $f_1, \dots, f_k: \N^n \to \N$, the pattern 
    \[
      f_1(x_1,\dots,x_n),\ldots, f_k(x_1,\dots,x_n)
    \]
    is PR if and only if there exist $\alpha_1,\dots ,\alpha_n\in\star\N$ such that \[f_1(\alpha_1,\dots,\alpha_n)\sim \dots\sim  f_k(\alpha_1,\dots,\alpha_n)\notin \mathbb N.\]
\end{fact}
The requirement that $f_i(\bla \alpha1,n)\notin \mathbb N$ corresponds to the nontriviality requirement in \Cref{notation}.\ref{point:nontrivialsol}.

\begin{rem}\label{rem:overQ}
  In what follows we will be interested in configurations of the form
  \begin{equation}
    \label{eq:31}
    \alpha\sim \beta\sim f_1(\alpha,\beta)\sim f_2(\alpha,\beta)
  \end{equation}
  where $f_1,f_2$ are certain polynomials over $\mathbb Q$, hence it will sometimes be convenient to work in $\star\mathbb Q$ instead of $\star \mathbb N$. 
The analogue of \Cref{thm:NScharPR} still holds (as it does over any set), but the reader should keep in mind that finding $\alpha,\beta\in \star\mathbb Q$ satisfying~\eqref{eq:31} only gives partition regularity of the configuration over $\mathbb Q$. Obtaining partition regularity over $\mathbb N$ requires to also prove that  $\alpha\in\star \mathbb N$ (in which case, all other terms in the pattern also lie in $\star \mathbb N$, since they are equivalent to $\alpha$).
\end{rem}

\Cref{pr:PropSim,pr:ClassiArchimedee,thm:NScharPR} are essentially all the nonstandard analysis that we will use in our proofs. Moreover, we will use a small amount of $p$-adic methods, recalled below.

\begin{rem}\*
    \begin{enumerate}
        \item In what follows, we will frequently fix a ``sufficiently large'' prime $p$. The ``sufficiently large'' is to be understood with respect to the height\footnote{Recall that, if $q=n/m$ with $(n,m)=1$, the \emph{height} of $q$ is $\max\set{\abs n,\abs m}$. In particular, the height of an integer is its absolute value.}  of the data of the problem at hand; in the case of \Cref{thm:qcd}, these would be $c,d,q$.  In particular, if $a,b\in \Z$ are small with respect to $p$, then $a\equiv b \pmod p$ implies $a=b$. This kind of arguments will be used repeatedly in the paper.
        \item Given a positive rational $q=k/h$ and a sufficiently large $p$, when we refer to the class of $q$ modulo $p$ we mean the class of $kh\inverse$, where the multiplicative inverse of $h$ is computed in $\mathbb{F}_p$.
        \item Besides working in the finite field $\mathbb F_p$ in the way just mentioned, some argument will use the ring of $p$-adic integers $\mathbb Z_p$, or its field of fractions $\mathbb Q_p$. In fact, we could work in these structures also in the proofs using $\mathbb F_p$ mentioned above, but we prefer to limit their use to a minimum, for the benefit of the reader unfamiliar with $\mathbb Z_p$. Intuitively, working in $\mathbb F_p$ amounts to considering the last digit of the expansion in base $p$, and working in $\mathbb Z_p$ to considering the last $\omega$ digits of said expansion; that is, if $\alpha=\sum_{i\in \star \mathbb N\cup \set 0} a_i p^i$, the class of $\alpha$ in $\mathbb Z_p$ may be identified with the sequence $(a_i)_{i\in \mathbb N\cup \set 0}$.
        \item We will make frequent use of the $p$-adic valuation $v_p$ and of the function $\smod_p$  sending $n$ to the class of $n/p^{v_p(n)}$ in $\mathbb F_p^\times$, i.e., to its least significant nonzero digit in base $p$. Recall that this map is a multiplicative homomorphism. When $p$ is clear from context, we will simply write $v$ and $\smod$ respectively.
        \item In order not to overburden the notation, we write e.g.\ $h\smod(\alpha)=k\smod (\beta)$ in place of $h\smod(\alpha)\equiv k(\smod \beta) \pmod p$.
    \end{enumerate}
  \end{rem}
The following observations will be crucial, and used throughout the paper, sometimes without mention. 
\begin{rem}\label{rem:smoddet}Let $p$ be a prime.
  \begin{enumerate}
  \item Since $\smod_p$ has finite image, it follows from points~\ref{point:propsim1} and~\ref{point:propsim4} of \Cref{pr:PropSim} that if $\alpha\sim \beta$ then $\smod_p(\alpha)=\smod_p(\beta).$
  \item Similarly, there is natural map $\star\mathbb N\to \mathbb Z_p$, sending $\alpha$ to the sequence of its last $\omega$ digits modulo $p$ (or to the sequence of its remainder classes modulo the standard powers of $p$, depending on the reader's favourite mental picture of $\mathbb Z_p$).

This map factors through the quotient by $\sim$. In other words, if $\alpha,\beta\in \star \mathbb N$ and $\alpha\sim \beta$, then $\alpha,\beta$ have the same class in $\mathbb Z_p$.
  \end{enumerate}
  \end{rem}

  \begin{lemma}\label{lemma:smodcases}
Let  $h\in \mathbb N$, $k,\ell \in \mathbb Z$,   $\eta,\zeta,\xi\in \star \mathbb N$,  $\theta\in \star \mathbb N\cup \set 0$, and $p>h+\abs{k}+\abs{\ell}$  a prime. Assume that $\eta\sim \zeta\sim \xi$, that  $h\eta\sim  \theta+k \zeta+\ell\xi$, and that $v_p(\theta)>v_p(\zeta),v_p(\xi)$.
\begin{enumerate}
\item If $v_p(\zeta)<v_p(\xi)$, then either $k=0$ or $k=h$.
\item If $v_p(\zeta)=v_p(\xi)$ then either  $k+\ell=0$ or $k+\ell=h$.
\end{enumerate}
  \end{lemma}
  \begin{proof}
Apply the map $\smod=\smod_p$ to $h\eta$ and $\theta+k \zeta+\ell\xi$, and observe that $\smod(\eta)=\smod(\zeta)=\smod(\xi)$, call it $s$.

    If $v_p(\zeta)<v_p(\xi)$, then either $k=0$ or $v(\theta+\ell \xi)>v(k\zeta)$, from which it follows that the least significant nonzero digit in base $p$ of $\theta+k \zeta+\ell\xi$ equals that of $k \zeta$, that is, $\smod(\theta+k \zeta+\ell\xi)=\smod(k\zeta)=ks$. This implies $\smod(h\eta)=hs=ks$, and as $p$ is large enough we obtain $h=k$.

    If $v_p(\zeta)=v_p(\xi)$ then either $k+\ell=0$ or   $\smod(\theta+k \zeta+\ell\xi)=\smod(k\zeta+\ell\xi)=(k+\ell)s$, and we conclude similarly as above.
   \end{proof}

   As a first example of application of the methods, we prove below \Cref{thm:qcd}.

\begin{proof}[Proof of \Cref{thm:qcd}]\label{proof:qcd}
Right to left is a weakening of \Cref{co:fsratiosrescaled}, so we focus on left to right.

  By \Cref{thm:NScharPR} there exist $\alpha,\beta\in \star \N\setminus \mathbb N$ such that $\alpha\sim \beta\sim c\alpha + d \beta\sim q \beta/\alpha$. In particular $q\beta/\alpha$ is a positive infinite integer, so the whole Archimedean class of $q\beta$, hence of $\beta$, must be larger than that of $\alpha$. Therefore, the Archimedean class of $c\alpha+d\beta$ equals that of $\beta$, and it follows from \Cref{pr:ClassiArchimedee} and $\beta\sim c\alpha+d\beta$ that $d=1$.

Work modulo a sufficiently large prime $p$ and recall \Cref{rem:smoddet}. Since $q\beta/\alpha\in \star \mathbb Z$ and $p$ is large, and in particular $v(q)=0$, we have  $v(\beta)\ge v(\alpha)$. By \Cref{lemma:smodcases} (applied with $\theta=0$) and the assumption $c\ne 0$ we see that either $c=1$, and we are done, or $v(\beta)=v(\alpha)$ and $c=-1$. In this last case, because $v(\alpha)=v(\beta)$ and $\smod(\alpha)=\smod(\beta)$, we obtain the contradiction
\[
0= v(q)+v(\beta)-v(\alpha) =v(q\beta/\alpha) \sim v(\beta+c\alpha)=v(\beta-\alpha)>0.\qedhere
\]
\end{proof}

As \Cref{thm:NScharPR,pr:PropSim} will be used very often, we will no longer mention them explicitly.

\section{Products of two linear polynomials}\label{sec:prodshifts}

  As a special case of~\eqref{eq:2}, for $\abs F\le 2$ we recover  the main result from \cite{goswamiMonochromaticTranslatedProduct2024} (which answers \cite[Question~31]{sahasrabudheExponentialPatternsArithmetic2018}), namely, that the pattern $x,y,xy,x+xy$ is PR. In fact, the case $\abs F\le 2$ of~\eqref{eq:2} provides a Ramsey version of this.

With $\abs F=3$ we obtain (Ramsey) partition regularity of the pattern
\[
x,\quad y,\quad z,\quad xy,\quad yz,\quad xyz,\quad x+xy,\quad  x+xyz,\quad  y+yz,\quad  xy+xyz,\quad   x+xy+xyz.
\]

As $x,y,xy, x(y+1)$ is PR, one may wonder whether other shifts of factors in this kind of configuration yield PR patterns. More precisely, we would like to determine for which $q\in \mathbb Q_{>0}$, $a,b\in \mathbb N$ and $n,m\in \mathbb Z$ the equation $z=q(ax+n)(by+m)$ is PR. The case $n=m=0$ follows from \Cref{co:fsratiosrescaled} and a change of variables, but can also be settled directly by the following observation.
  \begin{rem}\label{rem:qxy}
  The pattern $x,y, q xy$ is PR. To prove it, it suffices to take any multiplicative idempotent $u$ containing every $n \mathbb N$ and consider the ultrafilter $q\inverse\cdot u$.
\end{rem}
When at least one of $n,m$ is nonzero, the situation is more involved. Let us write $q=k/h$ and rephrase the problem, asking for which  $h,k,a,b\in \mathbb N$ and $n,m\in \mathbb Z$ the configuration
  \[
    hx,\quad  hy,\quad  k( a x+n)(by+m)
  \]
is PR, equivalently, there are $\alpha,\beta\in \star \mathbb N\setminus \mathbb N$ with
  \[
    h\alpha\sim h \beta\sim k( a\alpha+n)(b\beta+m).
  \]

We first deal with the case where exactly one of $n,m$ is nonzero, say $m$. By simple algebraic manipulations we may reduce to the case of patterns $hx,hy,  a x(by+m)$ where $(h,a)=1=(b,m)$.
  
  \begin{pr}\label{pr:lhdn}
Assume that $(h,a)=1=(b,m)$. If the pattern $hx,hy,  a x(by+m)$ is PR then  $a=1$.
    \end{pr}
    \begin{proof}    
 If not, let $p$ be a prime dividing $ a$, and let $v$ be the $p$-adic valuation.  As $(h, a)=1$, we have $v(h)=0$. Therefore, if $\alpha,\beta$ are witnesses of partition regularity,
\[
 v(\alpha)= v(h\alpha)\sim v( a \alpha( b\beta+ m))=v( a)+v(\alpha)+v( b\beta+ m).
\]
If $v(\alpha)\in \mathbb N$, then the $\sim$ above is an equality, and as $v( a)>0$ it follows that $v(\alpha)>v(\alpha)$, a contradiction. If instead $v(\alpha)> \mathbb N$ then $v(\beta)>\mathbb N$ as well, hence $v(b\beta+m)=v(m)$. Therefore, $v(\alpha)\sim v(a)+v(\alpha)+v(m)$, hence the function $x\mapsto x+v(a)+ v(m)$ sends the infinite number $v(\alpha)$ to an equivalent one, contradiction.
\end{proof}

We now proceed to deal with the case $n,m\ne 0$ and prove \Cref{athm:cosq}. In its proof, and in several arguments further down in the paper, we will use the following easy observations.

\begin{lemma}\label{lemma:fpotmt}
  Let $p$ a prime, $u,v\in \mathbb Z$ and $\alpha\in \star \mathbb N$.
  \begin{enumerate}
  \item If  the class $r$ of $\alpha$ in $\mathbb F_p$ satisfies $ur=v$ and  $\gamma\coloneqq u\alpha-v$, then $v_p(\gamma)> 0$.
  \item If the class $\rho$ of $\alpha$ in $\mathbb Z_p$ satisfies $u\rho=v$ and $\gamma\coloneqq u\alpha-v$, then $v_p(\gamma)> \mathbb N$.
  \end{enumerate}
\end{lemma}
\begin{proof}
By construction, the class of $\gamma$ in $\mathbb F_p$ [resp., $\mathbb Z_p$] is $0$. This means precisely that $\gamma$ is divisible by $p$ [resp., every standard power of $p$]. 
\end{proof}
\begin{rem}\label{rem:cofpmt}
Let $u,v\in \mathbb Z$, $f\in \mathbb Q[x,y]$, and $\alpha,\beta\in\star \mathbb N$ satisfy $\alpha\sim \beta\sim f(\alpha,\beta)$. Then $\gamma\coloneqq u\alpha-v$ and $\delta\coloneqq u\beta-v$ satisfy $\gamma\sim \delta\sim g(\gamma,\delta)$ for a suitable $g\in \mathbb Q[x,y]$. Suppose that $p$ is a prime sufficiently large with respect to $u,v$ and the coefficients of $g$. If the class $r$ of $\alpha,\beta$ in  $\mathbb F_p$ satisfies $ur=v$, it follows from \Cref{lemma:fpotmt} that the constant term of $g$ is null. Hence, we do not need to calculate explicitly the constant term of $g$ when applying this kind of transformation.
\end{rem}
\begin{thm}[\Cref{athm:cosq}]\label{thm:squares}  Let $q\in \mathbb Q_{>0}$, $a,b\in \mathbb N$ and $n,m\in \mathbb Z\setminus \set 0$.  Assume partition regularity of the pattern
  \[
    x,\quad  y,\quad  q( a x+n )(by+m).
  \]
  Then there is $t\in \mathbb Z$ such that
  \begin{enumerate}[label=(\alph*)]
  \item\label{case:conssqr} $q a m=t^2$ and $qbn=(t+1)^2$, or
  \item\label{case:consprod} $q a m=t(t-1)$ and $qbn=t(t+1)$.
  \end{enumerate}
\end{thm}

\begin{rem}\label{rem:thmbsym}
  Note that, by replacing $t$ with $-t-1$ or $-t$ respectively, one sees that both configurations in the conclusion are symmetrical in $x,y$.
\end{rem}

In the proof of \Cref{thm:squares}, we will use the following standard arithmetic fact, of which we provide a proof for the reader's convenience.
\begin{fact}\label{fact:gw}
  An integer $\Delta\in \mathbb Z$ is a square in $\mathbb Z$ if and only if, for cofinitely many $p$, its residue is a square in $\mathbb F_p$.
\end{fact}
\begin{proof}
  If $\Delta=0$ this is true, so assume this is not the case.
  
  By Gauss' Lemma, $\Delta$ is a square in $\mathbb Z$ if and only if it is a square in $\mathbb Q$. By the Grunwald--Wang Theorem, this holds if and only $\Delta$  is a square in $\mathbb Q_p$ for cofinitely many $p$. Equivalently, if and only if it is a square in $\mathbb Z_p$ for cofinitely many $p$ (since $v(x^2)=2v(x)$ and $\mathbb Z_p=\set{x\in \mathbb Q_p \mid v(x)\ge 0}$).

  By assumption, for cofinitely many $p$, the residue of $\Delta$ in $\mathbb F_p$ is a square, say $z^2$. The derivative of the polynomial $x^2-\Delta$ is $2x$. As $p$ is large enough, $p\nmid \Delta$, so $z\ne 0$, hence $2z \ne 0$. By  Hensel's Lemma, $\Delta$ is a square in $\mathbb Z_p$.
\end{proof}

\begin{proof}[Proof of \Cref{thm:squares}]
  Write $q=k/h$, for suitable positive integers $k,h$.
By replacing $a$ and $n$ by $ka$ and $kn$ respectively, we may assume that $k=1$.
Assume that $\alpha,\beta\in \star \mathbb N$ are infinite such that 
\begin{equation*}
  h\alpha\sim h \beta\sim ( a\alpha+n)(b\beta+m).
\end{equation*}

Fix a sufficiently large $p$. We have $v(\alpha)=v(h\alpha)\sim v(( a\alpha+n)(b\beta+m))$ which, if $v(\alpha)>0$ (hence also $v(\beta)>0$), results in the contradiction $v(\alpha)\sim v(n)+v(m)=0$. Therefore, we must have $v(\alpha)=0=v(\beta)$.

  Let $r$ be the residue class of $\alpha$ in $\mathbb F_p$. 
We have
\begin{equation}
  \label{eq:8}
  hr= a br^2 +( a m+bn)r+nm.
\end{equation}
View this as a degree $2$ equation in $r$, and let $\Delta \coloneqq( a m+bn-h)^2-4 a bnm$ be its discriminant. Because~\eqref{eq:8} has a solution in $\mathbb F_p$ (namely, $r$), its discriminant $\Delta$ must be a square in $\mathbb F_p$. As this happens for every sufficiently large $p$, by \Cref{fact:gw}, it follows that $\Delta$  is a square in $\mathbb Z$.

Set $A\coloneqq  a m$ and $B\coloneqq bn$. Observe that, since $n,m\ne 0$ by assumption, we also have $A,B\ne 0$. A routine calculation shows $\Delta=(A-B-h)^2-4Bh$. Let $L\in \mathbb Z$ be such that $\Delta=((A-B-h)+L)^2$. From $(A-B-h)^2-4Bh=((A-B-h)+L)^2$ it follows that $L(L-2h)=B(2L-4h)-2LA$, so $L$ is even. Set $\ell\coloneqq L/2$ and observe that the previous equality yields
  \begin{eqnarray}
    \label{eq:6}
    \ell(\ell-h)=B(\ell-h)-A\ell.
  \end{eqnarray}
In particular
we have $\ell\ne 0$ and $\ell\ne h$. View~\eqref{eq:6} as an affine equation in $A,B$, and parameterise its solutions as
  \[
    \begin{pmatrix}
      A\\B
    \end{pmatrix}
    =\begin{pmatrix}
      0\\\ell
     \end{pmatrix}
     +t\begin{pmatrix}
      \ell-h\\\ell
    \end{pmatrix}=
    \begin{pmatrix}
      t(\ell-h)\\(t+1)\ell
    \end{pmatrix}.
  \]
  By substituting this in $\Delta=((A-B-h)+L)^2$ we obtain $\Delta=(\ell-(t+1)h)^2$.  It follows that the solutions of~\eqref{eq:8} are
  \[
    r_1=-\frac{\ell t}{ A B} \quad\text{ and }\quad r_2=\frac{(1+t)(h-\ell)}{ A B}.
  \]
Note that $\ell,t$ do not depend on the choice of $p$.  Therefore, we may assume that $p$ is large also with respect to $\ell$ and $t$.
  
  We first consider the solution $r_1$. We have
   $A B\alpha+\ell t \equiv 0\pmod p$.
  Set
  \[
    \gamma\coloneqq A B\alpha+\ell t,\qquad \delta\coloneqq A B\beta+\ell t.
  \]
  By substituting in the original pattern, multiplying by $A B$ and adding $h\ell t$,
  \[
    h\gamma\sim h\delta\sim \gamma\delta-ht\gamma+\ell \delta.
  \]

  By \Cref{lemma:fpotmt} we have that $v(\gamma),v(\delta)>0$. We apply \Cref{lemma:smodcases}.
  \begin{itemize}
  \item   If $v(\delta)>v(\gamma)$ then either $ht=0$ or $h=-ht$. Both are contradictory as they imply $A=0$ and $B=0$ respectively.

  \item   If $v(\gamma)>v(\delta)$ then $\ell=0$ or $\ell =h$, and again  this implies the contradiction $AB=0$.

  \item   Suppose now $v(\gamma)=v(\delta)$. If $\ell =ht$ then $ a m=A=t(\ell -h)=ht(t-1)$ and $bn=B=(t+1)\ell =ht(t+1)$, so then~\ref{case:consprod} holds. Otherwise, $\ell =h(t+1)$, so $ a m=A=t(\ell -h)=ht^2$ and $bn=B=\ell (t+1)=h(t+1)^2$, as in~\ref{case:conssqr}. 
  \end{itemize}

  We now consider the solution $r_2$. In this case
$A B\alpha\equiv (1+t)(h-\ell )\pmod p$.
  Set
  \[
    \gamma\coloneqq  A B \alpha+(\ell -h)(1+t),\qquad \delta\coloneqq  A B \beta+(\ell -h)(1+t).
  \]
  Similarly as above, we obtain
  \[
    h\gamma\sim h\delta\sim \gamma\delta+(h-\ell )\gamma+h(1+t)\delta.
  \]
 By \Cref{lemma:fpotmt} $v(\gamma),v(\delta)>0$. As above, we apply \Cref{lemma:smodcases}.
  \begin{itemize}
  \item If $v(\delta)>v(\gamma)$, then either $h=\ell$ or $h=h-\ell$, so $\ell=0$. Both imply $AB=0$, a contradiction.
  \item If $v(\gamma)>v(\delta)$ then either $h(1+t)=0$ or $h=h(1+t)$, so $ht=0$. Again, both imply the contradiction $AB=0$.
  \item Assume now that $v(\gamma)=v(\delta)$. If $(h-\ell)+h(1+t)=0$ then  $h(2+t)=\ell $, hence $ a m=A=t(\ell -h)=ht(t+1)$ and $bn=B=(t+1)\ell=h(t+1)(t+2)$, so by shifting $t$ we fall in case~\ref{case:consprod} and we are done. If instead $h=h-\ell+h(1+t)$ then  $h(1+t)=\ell $, that is $\Delta=0$, so $r_2=r_1$ and we fall back to the previous case.\qedhere
  \end{itemize}
\end{proof}

\begin{rem}
 One can check that, in case~\ref{case:conssqr}, either $t$ or $-t$ equals $(qbn-qam-1)/2$. Similarly, in case~\ref{case:consprod}, either $t$ or $-1-t$ equals $(qbn-qam)/2$. From this, one may derive certain polynomial relations between $qam$ and $qbn$. As we will never use this, we leave details to the reader.
\end{rem}

\begin{eg}
  There are no $h, n\in \mathbb N$ such that the equation $hz=(x+n)(y+n)$ is PR.
\end{eg}

Consider now the equations $z=(x+1)(y+4)$ and $z=(4x+1)(y+1)$. They both fall in case~\ref{case:conssqr} of \Cref{thm:squares} with $t=1$, hence the partition regularity of neither of them is excluded by this result. Nevertheless, while we leave it open whether $z=(x+1)(y+4)$ is PR (\Cref{q:3pc}), easy parity considerations show that $z=(4x+1)(y+1)$ is not.

 Our \Cref{athm:newC} identifies, amongst the configurations respecting the conclusion of \Cref{athm:cosq}, which ones admit obstructions of this sort. Specifically, if such a configuration is PR, then a certain combination of the parameters needs to be an integer and, if this is the case, then the pattern may be equivalently rewritten in a certain canonical form.

The proof of \Cref{athm:newC} is split in the two lemmas below, which are stated in a technical way that will find further use in \Cref{sec:twopos}. Here we need to work in $\star \mathbb Q$, and \Cref{rem:overQ} becomes relevant.

\begin{lemma}\label{lemma:witcasea}
  Let $t\in \mathbb Z$ be such that $qa m=t^2$ and $qbn=(t+1)^2$.  Let $\alpha_0,\beta_0\in \star \mathbb Q$. Define
  \begin{itemize}
  \item $d\coloneqq -(qa m+qbn-1)/2q a b \in \mathbb Q$.
  \item $\alpha_1\coloneqq q a b(\alpha_0-d)$ and $\beta_1\coloneqq q a b(\beta_0-d)$.
  \item $\alpha_2\coloneqq \alpha_1-t(t+1)$ and $\beta_2\coloneqq \beta_1-t(t+1)$.
  \end{itemize}
  The following hold.
  \begin{enumerate}[label=(\Alph*)]
  \item\label{point:a2toa0}  We have $\alpha_2=qa b \alpha_0$.

  \item \label{point:Ba} The following are equivalent.
  \begin{enumerate}[label=(\roman*)]
  \item\label{point:caseareductioni}  $\alpha_0\sim \beta_0\sim q( a \alpha_0+n)(b\beta_0+m)$. 
 \item \label{point:caseareductionii} $\alpha_1\sim\beta_1\sim \alpha_1\beta_1+ (t+1)\beta_1-t\alpha_1$.
 \item\label{point:caseareductioniii}  $\alpha_2\sim\beta_2\sim (\alpha_2+(t+1)^2)(\beta_2+t^2)$.
 \end{enumerate}
\item\label{point:Ca}  If the equivalent conditions in point~\ref{point:Ba} hold, the following are equivalent.
 \begin{enumerate}[label=(\arabic*)]
 \item\label{point:caseareduction1} $\alpha_0\in \star \mathbb N$.
 \item\label{point:caseareduction2} $d\in \mathbb Z$ and $\alpha_1\in \star \mathbb N$.
 \item\label{point:caseareduction3}  $d\in \mathbb Z$ and $\alpha_2\in \star \mathbb N$.
 \end{enumerate}
  \item\label{point:Da} If the conditions in points~\ref{point:Ba} and~\ref{point:Ca} hold, then $\alpha_1$ is divisible by every natural number.
\end{enumerate}
\end{lemma}
\begin{proof}
    We leave it to the reader to check part~\ref{point:a2toa0} and to  perform the changes of variables showing that~\ref{point:caseareductioni}, \ref{point:caseareductionii} and \ref{point:caseareductioniii} are equivalent. The equivalence $\ref{point:caseareduction2}\Leftrightarrow\ref{point:caseareduction3}$ is immediate by definition of $\alpha_2$ and $\beta_2$.

    For    $\ref{point:caseareduction1}\Rightarrow\ref{point:caseareduction2}$, observe that $d=-t(t+1)/q a b$. Let $p$ be an arbitrary prime and let $\rho$ be the common class of $\alpha_0,\beta_0$ in $\mathbb Z_p$. We have
  \[
    \rho=q( a \rho+n)(b\rho+m)= q a b\rho^2+(q a m+kbn)\rho+knm
  \]
  whence $q a b \rho^2+2t(t+1)\rho+qnm=0$. Solving with the usual quadratic formula in $\mathbb Q_p$, we obtain that $\Delta=0$ and that $\rho=-t(t+1)/q a b$, which equals $d$.
  If $d\notin \mathbb Z$, then there is a prime $p$ such that  $v_p(d)<0$, contradicting $d=\rho\in \mathbb Z_p$. By construction, $\alpha_0-d$ is divisible by every power of every prime, hence by every natural number. This implies that $\alpha_1\in \star \mathbb N$, completing the proof of~\ref{point:caseareduction2}, and that $\alpha_1$ is itself divisible by every natural number, that is, point~\ref{point:Da}.

For  $\ref{point:caseareduction2}\Rightarrow\ref{point:caseareduction1}$, it suffices to observe that, for every $p$, the class of $\alpha_1$ in $\mathbb Z_p$ satisfies $\rho=\rho^2+\rho$, that is, $\rho=0$. Therefore, $\alpha_0=\alpha_1/qab+d\in \star\mathbb N$.
\end{proof}

\begin{lemma}\label{lemma:witcaseb}
  Let $t\in \mathbb Z\setminus\set 0$ be such that $q a m=t(t-1)$ and $qbn=t(t+1)$.  Let $\alpha_0,\beta_0\in \star \mathbb Q$. Define
  \begin{itemize}
  \item $d\coloneqq -( a m + bn)/2 a b\in \mathbb Q$.
  \item $\alpha_1\coloneqq q a b(\alpha_0-d)$ and $\beta_1\coloneqq q a b(\beta_0-d)$.
  \item $\alpha_2\coloneqq \alpha_1-t^2$ and $\beta_2\coloneqq \beta_1-t^2$.
  \end{itemize}
  The following hold.
  \begin{enumerate}[label=(\Alph*)]
  \item\label{point:a2toa0}  We have $\alpha_2=q a b \alpha_0$.

  \item\label{point:Bb} The following are equivalent.
    \begin{enumerate}[label=(\roman*)]
  \item\label{point:casebreductioni}  $\alpha_0\sim \beta_0\sim q( a \alpha_0+n)(b\beta_0+m)$. 
 \item \label{point:casebreductionii} $\alpha_1\sim\beta_1\sim \alpha_1\beta_1+t\beta_1-t\alpha_1$.
 \item\label{point:casebreductioniii}  $\alpha_2\sim\beta_2\sim (\alpha_2+t(t+1))(\beta_2+t(t-1))$.
 \end{enumerate}
\item\label{point:Cb}  If the equivalent conditions in point~\ref{point:Bb} hold, the following are equivalent.
 \begin{enumerate}[label=(\arabic*)]
 \item\label{point:casebreduction1} $\alpha_0\in \star \mathbb N$.
 \item\label{point:casebreduction2} $d\in \mathbb Z$ and $\alpha_1\in \star \mathbb N$.
 \item\label{point:casebreduction3}  $d\in \mathbb Z$ and $\alpha_2\in \star \mathbb N$.
\end{enumerate}
\item \label{point:Db} If the conditions in points~\ref{point:Ba} and~\ref{point:Ca} hold, then $\alpha_1$ is divisible by every natural number.
\end{enumerate}
\end{lemma}
\begin{proof}
As in the proof of \Cref{lemma:witcasea}, we prove $\ref{point:casebreduction1}\Leftrightarrow\ref{point:casebreduction2}$, as well as point~\ref{point:Db}, and leave it to the reader to check the remaining parts.
Write $q=k/h$, for suitable coprime $k,h\in \mathbb N$.

For the implication  $\ref{point:casebreduction1}\Rightarrow\ref{point:casebreduction2}$, if $p$ is any prime and $\rho$ is the common class of $\alpha_0,\beta_0$ in $\mathbb Z_p$, by solving the degree $2$ equation given by \ref{point:casebreductioni}
in $\mathbb Q_p$, we see that $\rho$ must satisfy $k a b \rho + ht^2=0$ or $k a b \rho + (t^2-1)h=0$.
  \begin{claim}\label{claim:rhoAp}
    For every prime $p$ we are in the case $k a b \rho+ht^2=0$.
  \end{claim}
  \begin{claimproof}
    In fact, if $p$ is such that the other case holds, set $\gamma\coloneqq k a b \alpha_0 + (t^2-1)h$ and $\delta\coloneqq k a b \beta_0 + (t^2-1)h$ and observe that $v_p(\gamma), v_p(\delta)$ are both infinite by \Cref{lemma:fpotmt}. A routine calculation shows that~\ref{point:casebreductioni} implies
    \begin{equation}
      \label{eq:16}
      h\gamma\sim h\delta\sim \gamma\delta+h(1-t)\gamma+h(1+t)\delta.
    \end{equation}

    Let $s\coloneqq\smod(\gamma)$. We have three cases, each split into various subcases.

    Assume first that $t=1$. Then~\eqref{eq:16} gives $h\delta\sim \gamma\delta+ 2h\delta$. If $p=2$, the $2$-adic valuation of the left hand side is $v(h)+v(\delta)$, while that of the right hand side is $1+v(h)+v(\delta)$, a contradiction since there are no equivalent points at finite nonzero distance. If $p\ne 2$, then $\smod(\gamma\delta+ 2h\delta)=2\smod(h)s$, therefore $\smod(h)s=2\smod(h)s$, again a contradiction.

    The case $t=-1$ is analogous, so we now assume that $t$ is neither $1$ nor $-1$ (nor $0$, by assumption).

     Let $\smod_{p^\omega} \from \star\mathbb N\to \mathbb Z_p\setminus\set 0$ be the function sending $x$ to the class in $\mathbb Z_p$ of $x/p^{v_p(x)}$. Intuitively, $\smod_{p^\omega}(x)$ is calculated by writing $x$ in base $p$, discarding the rightmost $0$ digits and taking the remaining rightmost $\omega$ digits.

 If $v_p(\gamma)<v_p(\delta)$, then $v_p(\delta)-v_p(\gamma)>\mathbb N$. This together with~\eqref{eq:16} implies $v(h)+v(\gamma)\sim v(h)+v(1-t)+v(\gamma)$, hence $v(1-t)=0$, so $p\nmid 1-t$. By applying $\smod_{p^\omega}$ we then find $\smod_{p^\omega}(\gamma)=\smod_{p^\omega}(1-t)\smod_{p^\omega}(\gamma)$, hence $\smod_{p^\omega}(1-t)=1$, which implies $t=0$, a contradiction.

The argument in the case $v_p(\gamma)>v_p(\delta)$ is completely analogous, and left to the reader.

    If $v_p(\gamma)=v_p(\delta)$, let us rewrite~\eqref{eq:16} as
    \begin{equation*}      
h\gamma\sim h\delta\sim      \gamma\delta+ht(\delta-\gamma)+h(\gamma+\delta).
    \end{equation*}
Observe that, because $\smod_{p^\omega}(\gamma)=\smod_{p^\omega}(\delta)$ and $v(\gamma)=v(\delta)$, we have $v(\delta-\gamma)-v(\gamma+\delta)>\mathbb N$. If $p=2$, by checking digits in base $2$ we see that $v(\gamma)\sim v(\gamma+\delta)=v(\gamma)+1$, a contradiction. If $p\ne 2$, we obtain $\smod(\gamma)=2\smod(\gamma)$, again a contradiction.
\end{claimproof}
To show that $d\in \mathbb Z$, observe that by \Cref{claim:rhoAp} and definition of $d$ we have $d=\rho$. Hence, for every $p$, the $p$-adic valuation of $d$ is nonnegative, and we have the conclusion.

It remains to verify that $\alpha_1\in \star \mathbb N$.
But, by construction and the fact that $d=\rho$, we see that $\alpha_0-d$ and $\beta_0-d$ are divisible by every element of $\mathbb N$, and we have the conclusion, as well as point~\ref{point:Db}.

For $\ref{point:casebreduction2}\Rightarrow\ref{point:casebreduction1}$, since $\alpha_0=\alpha_1/qab+d$, it suffices to show that $\alpha_1$ is divisible by every natural number. For every prime $p$, the class $\rho$ in $\mathbb Z_p$ of $\alpha_1$ satisfies $\rho=\rho^2$, hence it must equal $0$ or $1$. We show that the second case never happens. In fact, if $p$ is a counterexample and we set $\gamma\coloneqq \alpha_1-1$ and $\delta\coloneqq \beta_1-1$, we obtain $\gamma\sim \delta\sim \gamma\delta+(1-t)\gamma+(t+1)\delta$. We conclude by essentially the same proof as that of \Cref{claim:rhoAp}.
\end{proof}

\begin{rem}\label{rem:symchange}
In \Cref{lemma:witcasea,lemma:witcaseb}, if we replace $t$ by its symmetric parameter given by \Cref{rem:thmbsym}, the changes of variables relating $\alpha_0$ to $\alpha_1$, $\alpha_2$ remain unchanged.
\end{rem}

\begin{proof}[Proof of \Cref{athm:newC}]
  By \Cref{lemma:witcasea,lemma:witcaseb} we are only left to deal with the case where $qam=qbn=0$. But this is \Cref{rem:qxy}.
\end{proof}

\begin{eg}\label{thm:nld}
By \Cref{athm:newC} and partition regularity of $z=x(y+1)$ (\Cref{co:goswami}), the equation $z=x(18y+1)$ is PR.
More generally, for every coprime $b,m\in \mathbb N$, we have partition regularity of the pattern
  \[
    mx,\quad  my,\quad   x(by+m).
  \]
\end{eg}

\section{Several products of two linear polynomials}\label{sec:twopos}

In this section we study configurations of the form
\begin{equation}
  \label{eq:24}
  x,\quad   y,\quad  q_1(a_1 x+n_1)(b_1y+m_1),\quad  q_2(a_2x+n_2)(b_2y+m_2)
\end{equation}
for $q_i\in \mathbb Q_{>0}$, $a_i,b_i\in \mathbb N$ and $m_i, n_i\in \mathbb Z$, with the goal of proving \Cref{athm:4pieces,athm:div} and \Cref{aco:5pezzi}.

For patterns where all constant terms are nonzero, \Cref{athm:4pieces} may be proven by a quick argument using \Cref{athm:cosq}. Even if we do not know whether the conclusion of the latter holds when some constant term is null, we will still be able to prove \Cref{athm:4pieces} by slightly different arguments.

We begin by an easy observation that will be used without mention throughout the rest of the section.
\begin{lemma}\label{lemma:qabqab}
  If the pattern \eqref{eq:24} is PR, then $q_1a_1b_1=q_2a_2b_2$.
\end{lemma}
\begin{proof}
  Let $\alpha,\beta$ be witnesses of partition regularity. Observe that $q_i(a_i\alpha+n_i)(b_i\beta+m_i)$ is asymptotic to $q_ia_ib_i \alpha\beta$, that is, their ratio has standard part $1$. By \Cref{pr:ClassiArchimedee}
  \[
\st\left(\frac{q_2(a_2\alpha+n_2)(b_2\beta+m_2)}{q_1(a_1\alpha+n_1)(b_1\beta+m_1)}\right)=1
  \]
  hence $1=\st(q_2a_2b_2\alpha\beta/q_1a_1b_1\alpha\beta)$ and the conclusion follows.
\end{proof}
\begin{pr}\label{pr:thmcne0}
  In the case $m_1n_1m_2n_2\ne 0$, the conclusion of \Cref{athm:4pieces} holds.
\end{pr}
\begin{proof}
  Assume that~\eqref{eq:24} is PR and let $\alpha_0,\beta_0$ be witnesses.  Apply \Cref{athm:cosq} to the two 3-piece patterns $x,y,q_i(a_i x+n_i)(b_i y+m_i)$ for $i=1$ and $i=2$, obtaining $t_1,t_2\in \mathbb Z$ witnessing its conclusion. By \Cref{lemma:qabqab} we have $q_1a_1b_1=q_2a_2b_2$.

  It follows that the changes of variables in \Cref{lemma:witcasea,lemma:witcaseb} bringing $\alpha_0$ to $\alpha_2= q_1a_1b_1\alpha_0$ coincide. Fix a prime $p$ sufficiently large with respect to $t_1,t_2$. By part~\ref{point:Da} of  \Cref{lemma:witcasea,lemma:witcaseb} the class of $\alpha_2$ in $\mathbb F_p$ is either $-t_i(t_i+1)$ (if we are in case~\ref{case:conssqrintro} of  \Cref{athm:cosq}) or $-t_i^2$ (if we are in case~\ref{case:consprodintro}).
  \begin{itemize}
  \item If we are in case~\ref{case:conssqrintro} for both $i=1$ and $i=2$ then we get $t_1(t_1+1)=t_2(t_2+1)$. It follows that either $t_1=t_2$, hence we are in case~\ref{case:really3} of the conclusion, or $t_1=-t_2-1$, hence we are in case~\ref{case:cross}.

  \item If we are in case~\ref{case:consprodintro} for both $i=1$ and $i=2$ then we get $t_1^2=t_2^2$. It follows that either $t_1=t_2$, hence we are in case~\ref{case:really3} of the conclusion, or $t_1=-t_2$, hence we are in case~\ref{case:cross}.
   \item In the remaining cases we get $t_1(t_1+1)=t_2^2$ or $t_1^2=t_2(t_2+1)$, hence $t_1=0=t_2$. As all $q_i,a_i,b_i$ are nonzero, this implies that some $n_i$ or some $m_i$ is zero, against the assumptions\footnote{Note that, at any rate, this would bring us in case~\ref{case:0mixed} of the conclusion.}.\qedhere
  \end{itemize}
\end{proof}

We now deal with the case where $m_1n_1m_2n_2= 0$. The argument will be split into various subcases.

\begin{pr}\label{pr:nm0}
If the pattern~\eqref{eq:24} is PR and $n_1m_1=0$, then $n_2m_2=0$.
\end{pr}
\begin{proof}
  We may assume without loss of generality that $n_1=0$. Towards a contradiction, assume that $n_2m_2\ne 0$. Let  $t\in \mathbb Z$ be given by applying \Cref{athm:cosq} to the configuration obtained by ignoring the third piece.
  
  Fix a sufficiently large prime $p$ and let $r$ be the common class of $\alpha,\beta$ in $\mathbb F_p$. It follows from~\eqref{eq:24} that
  \begin{equation}
    \label{eq:25}
    r=q_1a_1r(b_1r+m_1)=q_2(a_2r+n_2)(b_2r+m_2).
  \end{equation}
  If $r=0$, then $q_2n_2m_2=0$, a contradiction. Therefore $r\ne 0$ and in particular
  \begin{equation}
    \label{eq:28}
    b_2r+m_2\ne 0.
  \end{equation}
By dividing the first equality in~\eqref{eq:25} by $r$, we then find that $1=q_1a_1b_1r+q_1a_1m_1$, hence that
\begin{equation}
  \label{eq:29}
  q_1a_1m_1\ne 1.
\end{equation}
  Let
  \[
    \gamma\coloneqq q_1a_1b_1\alpha-1+q_1a_1m_1,\qquad \delta\coloneqq q_1a_1b_1\beta-1+q_1a_1m_1.
  \]
As $v(\gamma)$ is positive by \Cref{lemma:fpotmt}, so is $v(\delta)\sim v(\gamma)$. We obtain from \eqref{eq:24} that\footnote{The reader may want to recall \Cref{rem:cofpmt}, especially in calculating the fourth piece of~\eqref{eq:27}.}
  \begin{equation}
    \label{eq:27}
    \gamma\sim \delta\sim \gamma\delta+\gamma+(1-q_1a_1m_1)\delta\sim \gamma\delta+(1-q_1a_1m_1+q_2a_2m_2)\gamma+(1-q_1a_1m_1+q_2b_2n_2)\delta.
  \end{equation}
  We have three cases, and in each we apply \Cref{lemma:smodcases}.

  \begin{itemize}
  \item   Assume $v(\delta)<v(\gamma)$. Applying \Cref{lemma:smodcases} to the second and third piece in~\eqref{eq:27}, we obtain that either $1=q_1a_1m_1$, contradicting~\eqref{eq:29}, or $1=1-q_1a_1m_1$, which implies $m_1=0$. By the same lemma applied to the second and fourth piece of \eqref{eq:27}, our assumptions give that either $1+q_2b_2n_2=0$, or $1=1+q_2b_2n_2$. In the latter case, we obtain the contradiction $q_2b_2n_2=0$. In the former, $q_2b_2n_2=-1$. By writing $q_2b_2n_2$ in terms of $t$, we get a contradiction since $-1$ is neither a square nor a product of two consecutive integers.

  \item   Assume $v(\gamma)<v(\delta)$. Comparing the first and fourth piece in~\eqref{eq:27} yields that either $1-q_1a_1m_1+q_2a_2m_2=0$, or $q_1a_1m_1=q_2a_2m_2$. In the second case, plugging the equality into~\eqref{eq:25} gives us $0=q_2b_2n_2r+q_2n_2m_2$, hence since $q_2n_2\ne 0$ we get $b_2r+m_2=0$, against~\eqref{eq:28}. The same contradiction can be obtained in the first case, by expanding the first equality in~\eqref{eq:25} and then substituting $q_2a_2b_2$ for $q_1a_1b_1$ and $1+q_2a_2m_2$ for $q_1a_1m_1$.

\item   Finally, assume $v(\gamma)=v(\delta)$. Comparing the first and third piece gives that  $q_1a_1m_1$ equals either $2$ or $1$, the latter being excluded by~\eqref{eq:29}. Similarly, comparing $\smod$ in the second and fourth piece and using that $q_1a_1m_1=2$ gives us that  $q_2a_2m_2+q_2b_2n_2$ equals either $2$ or $3$. Now observe that $q_2a_2m_2+q_2b_2n_2$ equals either $t^2+(t+1)^2$, or $2t^2$, depending on whether we are in case~\ref{case:conssqrintro} or in case~\ref{case:consprodintro} of \Cref{athm:cosq}. As neither $2$ nor $3$ is a sum of two consecutive squares, and as $2t^2=3$ has no solutions in $\mathbb Z$, it only remains to consider the case where $q_2a_2m_2=t(t-1)$, $q_2b_2n_2=t(t+1)$ and $t^2=1$. This gives either $m_2=0$ or $n_2=0$, and we are done.\qedhere
\end{itemize}
\end{proof}

In the case where $n_1=0=m_1$ we can fully characterise the PR patterns: they are either the $3$-piece pattern from \Cref{rem:qxy} or rescalings of the $4$-piece pattern in \Cref{co:goswami}.

 \begin{pr}\label{co:qxy}
   For $i\in \set{1,2}$, let $q_i\in \mathbb Q_{>0}$ and $a_i,b_i\in \mathbb N$. Let $n_2,m_2\in \mathbb Z$. The pattern
   \[
     x,\quad y,\quad q_1a_1b_1xy,\quad  q_2(a_2x+n_2)(b_2y+m_2)
   \]
   is PR if and only if $q_1a_1b_1=q_2a_2b_2$ and either
   \begin{enumerate}
   \item $n_2=m_2=0$, or 
   \item $n_2=0$ and $q_2a_2m_2=1$, or
   \item $m_2=0$ and $q_2b_2n_2=1$.
   \end{enumerate}
 \end{pr}
 \begin{proof}
 By \Cref{pr:ClassiArchimedee}, if the pattern is PR then  $q_1a_1b_1=q_2a_2b_2$, therefore we assume this in the rest of the proof. Let $q\coloneqq q_1a_1b_1=q_2a_2b_2$.

 By \Cref{pr:nm0}, if the pattern is PR then at least one of $n_2,m_2$ has to be zero.  If $n_2=m_2=0$, then the pattern is PR by \Cref{rem:qxy}.

 Let us deal with the case $n_2=0$, $m_2\ne 0$, the case $m_2=0$, $n_2\ne 0$, being symmetrical.   Note that, if $\alpha\sim\beta\sim q\alpha\beta\sim q \alpha\beta+q_2a_2m_2\alpha$ and $\alpha,\beta$ belong to $\star \mathbb N$, then so do the other items above. Therefore, $\alpha'\coloneqq \alpha$ and $\beta'\coloneqq q\alpha\beta$ belong to $\star \mathbb N$ and satisfy
 \begin{equation}
   \label{eq:30}
   \alpha'\sim\beta'/(q\alpha')\sim \beta'\sim \beta'+q_2a_2m_2\alpha'.
 \end{equation}
This is PR if and only if $q_2a_2m_2=1$ by \Cref{thm:qcd}. Conversely, if $q_2a_2m_2=1$ and $\alpha',\beta'\in \star \mathbb N$ satisfy~\eqref{eq:30} then we conclude by setting $\alpha\coloneqq\alpha'$ and $\beta\coloneqq \beta'/q\alpha'$, and observing that the latter lies in $\star \mathbb N$ since it is equivalent to $\alpha'$.
 \end{proof}

 \begin{eg} By \Cref{co:qxy}, we have that the configuration $2x,2y,3xy,x(3y+2)$ is PR, whilst $2x,2y,3xy,x(3y-2)$ is not.
 \end{eg}

 When $m_1\ne 0$ we are left with the cases $n_2=0\ne m_2$ and $m_2=0\ne n_2$. It turns out that in the first case there are no PR 4-piece patterns, and that PR in the second case implies symmetry in $x,y$.

 \begin{pr}\label{thm:mixed2to1}
 Let $m_1m_2\ne 0$.  If the pattern
   \[
     x,\quad  y,\quad q_1a_1 x(b_1y+m_1),\quad q_2a_2x(b_2y+m_2)
   \]
   is PR then $q_1a_1 x(b_1y+m_1)=q_2a_2x(b_2y+m_2)$.
 \end{pr}
 \begin{proof}
 As usual let $\alpha,\beta$ be witnesses of partition regularity and  observe that by \Cref{lemma:qabqab} $q_1a_1b_1=q_2a_2b_2$. The common class $r$ of $\alpha,\beta$ modulo a sufficiently large prime $p$ satisfies $r= q_1a_1b_1r^2+q_1a_1m_1r=q_1a_1b_1r^2+q_2a_2m_2r$, whence $q_1a_1m_1 r=q_2a_2m_2 r$. If $r\ne 0$, as $p$ is large, we obtain  $q_1a_1m_1=q_2a_2m_2$ and, as $q_1a_1b_1=q_2a_2b_2$, we obtain the conclusion.

 If instead $r=0$, then $v(\alpha),v(\beta)$ are positive. If $s\coloneqq\smod(\alpha)=\smod(\beta)$ we obtain $q_1a_1m_1s=q_2a_2m_2s$ and, again because $p$ is large, this means $q_1a_1m_1=q_2a_2m_2$, so we conclude as above.
 \end{proof}

 \begin{eg} \label{eg:nobrauer}
The configuration $x,y,x(y+1),x(y+2)$ is not PR. 
 \end{eg}

 \begin{rem}
   Recall that \emph{Brauer's Theorem} is the strengthening of Van der Waerden's Theorem asserting that in every colouring of the natural numbers one may find a colour containing arbitrarily long arithmetic progressions together with their common differences.

   It follows from \Cref{eg:nobrauer} that one may not add quotients to the configuration in Brauer's Theorem, in fact not even to the length 3 case. Indeed, the configuration $x,y,x+y,2x+y, y/x$ is not PR, since the change of variables $z\coloneqq y/x$ turns it into  $x, zx, x(z+1),x(z+2), z$, and $x,y,x+y,x+2y, y/x$ is not PR because of \Cref{pr:ClassiArchimedee}.
 \end{rem}

 \begin{pr}
   If $m_1n_2\ne 0$ and the pattern
   \[
     x,\quad y, \quad q_1a_1x(b_1y+m_1),\quad q_2(a_2x+n_2)b_2y
   \]
   is PR, then $q_1a_1m_1=q_2b_2n_2$.
 \end{pr}
 \begin{proof}
   Again, by \Cref{pr:ClassiArchimedee} partition regularity implies $q_1a_1b_1=q_2a_2b_2$.
   Let $\alpha,\beta$ be nonstandard witnesses of partition regularity and $r$ their common class modulo a sufficiently large prime $p$.
     We have $r=q_1a_1b_1r^2+q_1a_1m_1r=q_1a_1b_1r^2+q_2b_2n_2r$, 
   whence $q_1a_1m_1r=q_2b_2n_2r$. 
   If $r\ne 0$ we immediately get $q_1a_1m_1=q_2b_2n_2$.    If $r=0$, then $v(\alpha),v(\beta)$ are positive. As $p$ is large if $s\coloneqq \smod(\alpha)$, we obtain $q_1a_1m_1s=q_2n_2b_2s$, hence the conclusion.
 \end{proof}
 \begin{eg}
   The configuration $x,y,x(y+1),(x+2)y$ is not PR.
 \end{eg}
This was the last case to consider in the proof of \Cref{athm:4pieces}, and \Cref{aco:5pezzi} follows easily.
 \begin{proof}[Proof of \Cref{athm:4pieces}]
   By combining the previous results in this section.
 \end{proof}
 \begin{proof}[Proof of \Cref{aco:5pezzi}]
    Any PR configuration with 5 or more pieces induces various PR 4-piece configurations. By inspecting the conclusion of \Cref{athm:4pieces}, we see that the only compatible 4-piece configurations are those in case~\ref{case:0mixed}, with the same $q$, from which the first part of the conclusion follows. For the ``moreover part'', if $p$ is an arbitrary prime and $\alpha,\beta$ witness partition regularity of the pattern, then their class $\rho$  in $\mathbb Z_p$ satisfies $q\rho^2+\rho=q\rho^2$. Therefore, $\alpha$ and $\beta$ are divisible by every natural number, and by setting $\alpha'\coloneqq q\alpha$ and $\beta'\coloneqq q\beta$ we obtain witnesses of partition regularity of $x,y,xy,xy+x,xy+y$. The converse is analogous.
\end{proof}
Finally, we prove \Cref{athm:div}.

\begin{proof}[Proof of \Cref{athm:div}]
  By \Cref{lemma:witcasea,lemma:witcaseb} and \Cref{rem:symchange}, it suffices to show that we are in the assumptions of said lemmas. If $nm\ne 0$, this is \Cref{athm:cosq}.  If $nm=0$, since by assumption the third and fourth piece of the pattern differ, precisely one of $n,m$ is null, without loss of generality $n=0\ne m$. To conclude, we prove that $qam$ equals either $1$ or $2$, so we are in the assumptions of \Cref{lemma:witcasea} with $t=-1$ or of  \Cref{lemma:witcaseb} with $t=1$ respectively.

  Let $\alpha,\beta\in\star \mathbb N$ witness partition regularity. We move to integer parameters, by writing $q\cdot a\cdot(b,m)=A/h$ for suitable coprime integers $A,h$ and setting $B\coloneqq b/(b,m)$ and $M\coloneqq m/(b,m)$. By multiplying by $h$ we obtain
   \begin{equation*}
     h \alpha\sim h\beta\sim  A \alpha(B\beta+M)\sim  A (B\alpha+M)\beta
   \end{equation*}
   where $(h, A)=1=(B,M)$. After this renaming, the conclusion amounts to showing that $h=AM$ or $2h=AM$.

   Fix a sufficiently large prime $p$.  If $v(\alpha),v(\beta)>0$, by \Cref{lemma:smodcases} we obtain immediately $h= A M$.   If instead $v(\alpha)=0=v(\beta)$, set
   \[
     \gamma\coloneqq  A B\alpha+ A M-h\quad \text{ and } \quad    \delta\coloneqq  A B\beta+ A M-h.
   \]
Similarly as previously done in this paper, we obtain
   \begin{equation*}
     h\gamma\sim h\delta\sim \gamma\delta+h\gamma+(h- A M)\delta\sim  \gamma\delta+(h- A M)\gamma+h\delta
   \end{equation*}
 and $v(\gamma),v(\delta)>0$.
 If $v(\gamma)>v(\delta)$, we apply \Cref{lemma:smodcases} to the second and third piece of the pattern. If $h=h- A M$ then $ A M=0$, a contradiction, therefore $h-AM=0$. If $v(\gamma)<v(\delta)$ the argument is analogous, but we apply \Cref{lemma:smodcases} to the first and fourth piece, so we are left with the case $v(\gamma)=v(\delta)>0$. Again by \Cref{lemma:smodcases}, applied to the first and third piece, this implies either $2h= A M$ or $h=2h-AM$, that is, $h=AM$.
\end{proof}

 \begin{eg}
   For all $n\in\Z\setminus\{0,1,2\}$ the configuration $x,y,x(y+n),(x+n)y$ is not PR.
 \end{eg}

\section{Open problems}\label{sec:froq}

It remains open to determine which patterns obeying the restrictions given by our main theorems are PR. For 3-piece patterns, this amounts to answering the following.

\begin{question}\label{q:3pc}\*
  \begin{enumerate}
    \item For which $t\in \mathbb Z$ is the equation $z=(x+t^2)(y+(t+1)^2)$ PR?
    \item For which $t\in \mathbb Z$ is the equation $z=(x+t(t-1))(y+t(t+1))$  PR?
    \item\label{case:nocanform} For which $q\in \mathbb Q_{>0}$, $b\in \mathbb N$ and $m\in \mathbb Z$ is the equation  $z=qx(by+m)$ PR?
\end{enumerate}
\end{question}

The abundance of parameters in \Cref{q:3pc}.\ref{case:nocanform} is due to the lack of a canonical form when $n=0$.

\begin{problem}\label{prob:canform}
 Find a canonical form, preserving (non-)partition regularity, for equations of the form $z=qx(by+m)$.
\end{problem}

As for patterns with more than 3 pieces, our work reduces the general problem to the following.
\begin{question}\label{q:4or5pc}\*
  \begin{enumerate}
      \item For which $t\in \mathbb Z$ is  $x,y,(x+t^2)(y+(t+1)^2),(x+(t+1)^2)(y+t^2)$  PR?
    \item For which $t\in \mathbb Z$ is  $x,y,(x+t(t-1))(y+t(t+1)),(x+t(t+1))(y+t(t-1))$  PR?
    \item Is $x,y,xy,xy+x,xy+y$ PR?
    \end{enumerate}
  \end{question}

As the pattern $x,y,xy,x(y+1)$ is PR, one wonders if this is the case when we replace $x+y$ for $xy$.

\begin{question}\label{q:sumshift}
  Is the pattern $x,y,x+y,x(y+1)$ PR?
\end{question}

Let us observe that 1 above cannot be replaced by other nonzero integers. 

\begin{rem}
If $x,y,x+y,x(y+m)$ is PR then  $m=1$. This can be easily shown by fixing a large prime $p$, observing that the class $r$ of $x,y$ in $\mathbb F_p$ must satisfy $r=r+r$, hence $r=0$, colouring with values of $\smod$ and observing that $1=\smod(y+m)=\smod(m)$.
\end{rem}

All of the questions above also have Ramsey versions, again in the sense of \cite{dinassoRamseysWitnesses2025}.

\begin{problem}
  Study the variants of \Cref{q:3pc,q:4or5pc,q:sumshift} given by replacing ``PR'' by ``Ramsey PR''.
\end{problem}

\section{Appendix: some standard proofs}
\label{sec:appendix}

In this section we give standard versions of some previously mentioned statements and of the proof \Cref{thm:qcd}. We begin with a counterpart to \Cref{thm:NScharPR}.
\begin{fact}\label{rem:StandcharPR}
The configuration $f_1(\bla x1,k),\ldots, f_n(\bla x1,k)$ is partition regular if and only if there is $w\in \beta \mathbb N^k\setminus \mathbb N^k$ such that $f_1(w)=\ldots=f_n(w)\in \beta \mathbb N\setminus \mathbb N$.
\end{fact}
The following lemma is the translation in standard terms of the fact that equivalent nonstandard points must be at infinite distance. Let $\pi_1,\pi_2\from \mathbb N^2\to \mathbb N$ be the usual coordinate projections.
\begin{lemma}\label{lemma:mposv}
Let $w\in\beta\N^2$ be such that $\pi_1(w)=\pi_2(w)$,
and let $f:\N\to\N$. If $\{(x,y)\mid |f(x)-f(y)|\le \ell\}\in w$ for some $\ell\in\N$
then $\{(x,y)\mid f(x)=f(y)\}\in w$.
\end{lemma}

\begin{proof}
Let $X_k\coloneqq \{(x,y)\mid |f(x)-f(y)|=k\}$. Since $\bigcup_{k=0}^\ell X_k\in w$,
there exists $k$ such that $X_k\in w$. Assume for the sake of contradiction that $k\ne 0$,
and hence $k+1\ge 2$. 
For $s=0,\ldots,k$ let $C_s\coloneqq \{x\mid f(x)\equiv s\mod (k+1)\}$. 
Since $\N=\bigcup_{s=0}^k C_s$ there exists
$s$ such that $C_s\in\pi_1(w)=\pi_2(w)$.
But then $(C_s\times C_s)\in w$ while $(C_s\times C_s)\cap X_k=\emptyset$,
a contradiction.
\end{proof}

The standard reformulation of \Cref{pr:ClassiArchimedee} is the following.
\begin{lemma}\label{lemma:starchclass}
  Let $w\in \beta \mathbb N^2$ be such that $\pi_1(w)=\pi_2(w)$. If there is $n$ such that $\set{(x,y)\mid \frac 1n x<y<nx}\in w$, then for all $\epsilon>0$ we have $\set{(x,y)\mid \abs{y/x -1}<\epsilon}\in w$.
\end{lemma}
\begin{proof}
For $a\in\N$, let $L(a)$ be the length of the binary representation of $a$, that is, the natural number such that $2^{L(a)-1}\le a<2^{L(a)}$. 
A straightforward computation shows that
\[
  \text{if }\frac{1}{n}x<y<nx \text{ then } \abs{L(y)-L(x)}\le L(n).
\]
By this and assumption it follows that
$\{(x,y)\mid \abs{L(y)-L(x)}\le L(n)\}\in w$ and so, by \Cref{lemma:mposv},
$\Lambda\coloneqq \{(x,y)\mid L(x)=L(y)\}\in w$.
For $s\ge 1$, let $\lambda_s:\N\to\{0,1\}$ be the function 
where $\lambda_s(a)$ is the $s$-th digit in the binary expansion of $a$
(we agree that $\lambda_s(a)=0$ if $s>L(a)$).
Since $\{(x,y)\mid \abs{\lambda_s(x)-\lambda_s(y)}\le 1\}=\N\times\N\in w$,
again by \Cref{lemma:mposv} we obtain that 
$\Lambda_s\coloneqq \{(x,y)\mid \lambda_s(x)=\lambda_s(y)\}\in w$.

Pick $k\in\N$ such that $2^{-k}<\epsilon$. 
We will reach the conclusion by showing that
\begin{equation}
  \label{eq:23}
  [2^{k-1},+\infty)^2\cap\Lambda\cap\bigcap_{s=1}^k\Lambda_s\subseteq
  \set*{(x,y)\biggm| \abs*{\frac{x}{y}-1}<\epsilon}.
\end{equation}

Indeed, since $\pi_1(w)=\pi_2(w)$ is non-principal, 
$[2^{k-1},+\infty)^2\in w$; besides, we already noticed that
$\Lambda,\Lambda_1,\ldots,\Lambda_k\in w$.

Let us finally show~\eqref{eq:23}.
Pick any $(x,y)$ in the intersection on the left hand side.
Since $(x,y)\in\Lambda$, the numbers $x=\sum_{s=1}^{\ell}x_s 2^{\ell-s}$
and $y=\sum_{s=1}^{\ell}y_s 2^{\ell-s}$ have binary representations of the same length $\ell$.
Note that $x_s=y_s$ for $s=1,\ldots,k$,
since $(x,y)\in\bigcap_{s=1}^k\Lambda_s$.
If we let $z\coloneqq \sum_{s=1}^k x_s 2^{\ell-s}$, then $x'\coloneqq x-z=\sum_{s=k+1}^\ell x_s2^{\ell-s}<2^{\ell-k}$
and $y'\coloneqq y-z=\sum_{s=k+1}^\ell y_s2^{\ell-s}<2^{\ell-k}$, and we obtain the desired inequality:
\[\abs*{\frac{x}{y}-1}=\abs*{\frac{z+x'}{z+y'}-1}=\frac{\abs{x'-y'}}{z+y'}\le\frac{\max\{x',y'\}}{y}\le
\frac{2^{\ell-k}}{2^\ell}=\frac{1}{2^k}<\epsilon.\qedhere\]
\end{proof}

\begin{proof}[Standard proof of \Cref{thm:qcd}]
The proof of right to left is the same as in the one originally provided at page~\pageref{proof:qcd} (and does not use nonstandard methods).

Left to right, let $f(x,y)=cx+dy$ and $g(x,y)=q\cdot y/x$ if $q\cdot y/x$ is an integer, and $1$ otherwise.

By \Cref{rem:StandcharPR} there exists $w\in \beta \mathbb N^2\setminus \mathbb N^2$ such that $\pi_1(w)=\pi_2(w)=f(w)=g(w)$. As $g(w)$ is nonprincipal, it follows that, for every $n\in \mathbb N$, we have  $\set{(x,y)\mid y>nx}\in w$. Therefore, for $n=d+1$, we have that  $\set{(x,y)\mid \frac 1ny <cx+dy<ny}\in w$ and, by \Cref{lemma:starchclass} (applied to the pushforward of $w$ along $(x,y)\mapsto(y, cx+dy)$), it follows that for all $\epsilon>0$ we have  $\set{(x,y)\mid \abs{(cx+dy)/y -1}<\epsilon}\in w$; since for every $n$ we also have  $\set{(x,y)\mid y>nx}\in w$, it follows that $d=1$. 

    Fix a sufficiently large prime $p$, let $v$ be the $p$-adic valuation and $\smod$ the function sending $n$ to $(n/p^{v(n)}\mod p)\in \mathbb F_p^\times$. Observe that, as $\pi_1(w)=\pi_2(w)$, we have $(\smod\circ \pi_1)(w)=(\smod\circ \pi_2)(w)$. Moreover, as $\pi_1(w)=f(w)$ and $\smod$ has finite image, we have
    \begin{equation}
      \label{eq:34}
      \set{(x,y) \mid \smod(x)=\smod(cx+y) }\in w.
    \end{equation}

 Since $\set{(x,y)\mid qy/x \in \mathbb N)}\in w$ and $p$ is large, so $v(q)=0$, we have  $\set{(x,y)\mid v(y)\ge v(x)}\in w$. Write $\set{(x,y)\mid v(y)\ge v(x)}=\set{(x,y)\mid v(y)>v(x)}\cup \set{(x,y)\mid v(y)= v(x)}$.

 We consider two cases depending on what is in $w$. 
If $\set{(x,y)\mid v(y)> v(x)}\in w$, then because $p\nmid c$ we have 
$\set{(x,y) \mid \smod(cx+y)=\smod(cx)}\in w$. Recall that  $\smod(cx)=\smod(c)\smod(x)$, hence
\[
  \set{(x,y) \mid \smod(x)=\smod(c)\smod(x)}\in w
\]
    and it follows that $\smod (c)=1$.  Since $p$ is sufficiently large, this entails $c=1$.

    Assume now that $\set{(x,y)\mid v(y)= v(x)}\in w$.
    Note that, as $v(q)=0$, this set equals $\set{(x,y)\mid v(qy/x)=0}$. Since $f(w)=g(w)$, we have $\set{(x,y)\mid v(y+cx)=0}\in w$. Because $\set{(x,y)\mid\smod(y)=\smod(x)}\in w$, this implies $c\ne -1$, as otherwise $\set{(x,y)\mid v(y+cx)>0}\in w$.     By also using  that $v(c)=0$, it follows that $\set{(x,y) \mid \smod(cx+y)=(1+c)\smod(x)}\in w$, so by~\eqref{eq:34} $\set{(x,y)\mid \smod(x)=(1+c)\smod(x)}\in w$, hence $\set{(x,y) \mid c\smod(x)= 0}\in w$. This is a contradiction as $p\nmid c$ and $0$ is not in the image of $\smod$.
\end{proof}

\renewcommand*{\bibfont}{\normalfont\footnotesize}

\printbibliography

@misc{goswamiMonochromaticTranslatedProduct2024,
  title = {Monochromatic {{Translated Product}} and {{Answering Sahasrabudhe}}'s {{Conjecture}}},
  author = {Goswami, Sayan},
  year = 2024,
  number = {arXiv:2412.17868},
  eprint = {2412.17868},
  publisher = {arXiv},
  archiveprefix = {arXiv}
}

@article{dinassoRamseysWitnesses2025,
  title = {Ramsey's Witnesses},
  author = {Di Nasso, Mauro and Luperi Baglini, Lorenzo and Mamino, Marcello and Mennuni, Rosario and Ragosta, Mariaclara},
  year = {2026},
  journal = {Combinatorial Theory},
  volume = {(to appear)},
note = {Preprint available at arXiv: \href{https://arxiv.org/abs/2503.09246}{\texttt{2503.09246}}}
}

@article{dinassoSelfdivisibleUltrafiltersCongruences2025,
  title = {Self-Divisible Ultrafilters and Congruences in {{$\beta \mathbb Z$}}},
  author = {Di Nasso, Mauro and Luperi Baglini, Lorenzo and Mennuni, Rosario and Pierobon, Moreno and Ragosta, Mariaclara},
  year = {2025},
  journal = {The Journal of Symbolic Logic},
  volume = {90},
  number = {3},
  pages = {1180--1197},
  doi = {10.1017/jsl.2023.51},
}

@article{bergelsonPolynomialExtensionsMillikenTaylor2014,
  title = {Polynomial Extensions of the {{Milliken-Taylor Theorem}}},
  author = {Bergelson, Vitaly and Hindman, Neil and Williams, Kendall},
  year = {2014},
  journal = {Transactions of the American Mathematical Society},
  volume = {366},
  number = {11},
  pages = {5727--5748},
  doi = {10.1090/S0002-9947-2014-05958-8}
}

@article{sahasrabudheExponentialPatternsArithmetic2018,
  title = {{Exponential patterns in arithmetic Ramsey theory}},
  author = {Sahasrabudhe, Julian},
  year = {2018},
  journal = {Acta Arithmetica},
  volume = {182},
  pages = {13--42},
  publisher = {Instytut Matematyczny Polskiej Akademii Nauk},
  doi = {10.4064/aa8603-9-2017}
}

@phdthesis{baglini2012hyperintegers,
  title = {Hyperintegers and {{Nonstandard Techniques}} in {{Combinatorics}} of {{Numbers}}},
  author = {Luperi Baglini, Lorenzo},
  year = 2012,
  eprint = {1212.2049},
  archiveprefix = {arXiv},
  school = {Università di Siena}
}

@incollection{di2015hypernatural,
  title = {Hypernatural {{Numbers}} as {{Ultrafilters}}},
  booktitle = {Nonstandard {{Analysis}} for the {{Working Mathematician}}},
  author = {Di Nasso, Mauro},
  editor = {Loeb, Peter A. and Wolff, Manfred P. H.},
  year = {2015},
  pages = {443--474},
  publisher = {Springer Netherlands},
  address = {Dordrecht},
  doi = {10.1007/978-94-017-7327-0_11},
  langid = {english},
}

@article{di2018ramsey,
  title = {Ramsey Properties of Nonlinear {{Diophantine}} Equations},
  author = {Di Nasso, Mauro and Luperi Baglini, Lorenzo},
  year = 2018,
  journal = {Advances in Mathematics},
  volume = {324},
  pages = {84--117},
  doi = {10.1016/j.aim.2017.11.003}
}

@book{DGL,
  title = {Nonstandard {{Methods}} in {{Ramsey Theory}} and {{Combinatorial Number Theory}}},
  author = {Di Nasso, Mauro and Goldbring, Isaac and Lupini, Martino},
  year = 2019,
  series = {Lecture {{Notes}} in {{Mathematics}}},
  number = {2239},
  publisher = {Springer},
doi ={10.1007/978-3-030-17956-4},
}

@article{schurUberKongruenz,
  title = {Uber Die {{Kongruenz}}  $x^m+y^m=z^m \pmod p$},
  author = {Schur, Issai},
  year = {1916},
  journal = {Jahresbericht Der Deutschen Mathematiker-vereinigung},
  volume = {25},
  pages = {114--117}
}

@article{hindmanFiniteSumsSequences1974,
  title = {Finite Sums from Sequences within Cells of a Partition of {$N$}},
  author = {Hindman, Neil},
  year = {1974},
  journal = {Journal of Combinatorial Theory, Series A},
  volume = {17},
  number = {1},
  pages = {1--11},
  doi = {10.1016/0097-3165(74)90023-5},
}

@book{hindmanAlgebraStoneCechCompactification2011,
  title = {Algebra in the {{Stone-\v Cech Compactification}}: {{Theory}} and {{Applications}}},
  shorttitle = {Algebra in the {{Stone-\v Cech Compactification}}},
  author = {Hindman, Neil and Strauss, Dona},
  year = {2011},
  publisher = {De Gruyter},
  doi = {10.1515/9783110258356},
  note = {2nd revised and extended edition}
}

@preamble{ "\providecommand{\noopsort}[1]{} " }

@article{bowenMonochromaticProductsSums2024,
  title = {Monochromatic Products and Sums in the Rationals},
  author = {Bowen, Matt and Sabok, Marcin},
  year = 2024,
  journal = {Forum of Mathematics, Pi},
  volume = {12},
  pages = {e17},
  doi = {10.1017/fmp.2024.19}
}

@article{bowenMonochromaticProductsSums2025,
  title = {Monochromatic Products and Sums in 2-Colorings of {$\mathbb N$}},
  author = {Bowen, Matt},
  year = 2025,
  journal = {Advances in Mathematics},
  volume = {462},
  pages = {110095},
  doi = {10.1016/j.aim.2024.110095}
}

@article{moreiraMonochromaticSumsProducts2017,
  title = {Monochromatic Sums and Products in {$\mathbb N$}},
  author = {Moreira, Joel},
  year = 2017,
  journal = {Annals of Mathematics},
  volume = {185},
  number = {3},
  doi = {10.4007/annals.2017.185.3.10}
}

@article{hindmanPartitionsPairwiseSums1984a,
  title = {Partitions and Pairwise Sums and Products},
  author = {Hindman, Neil},
  year = {1984},
  journal = {Journal of Combinatorial Theory, Series A},
  volume = {37},
  number = {1},
  pages = {46--60},
  doi = {10.1016/0097-3165(84)90018-9},
}

@article{hindmanPartitionsSumsProducts1979,
  title = {Partitions and Sums and Products of Integers},
  author = {Hindman, Neil},
  year = 1979,
  journal = {Transactions of the American Mathematical Society},
  volume = {247},
  pages = {227--245},
  doi = {10.1090/S0002-9947-1979-0517693-4}
}

\end{document}